\documentclass[10pt]{article}
\usepackage{lineno,hyperref}

\usepackage{setspace}
\usepackage{epsfig}  
\usepackage{epstopdf}
\usepackage{hyperref}
\usepackage{chngcntr}

\usepackage[margin=1in]{geometry}
\usepackage{amsthm}
\usepackage{amssymb,amsmath, mathtools}
\usepackage{mathrsfs}

\usepackage{enumitem}
\usepackage{subfiles}
\usepackage{bm}

\usepackage{caption}
\usepackage{subcaption}
\usepackage{multirow}
\usepackage{booktabs}
\usepackage{cancel}

\makeatletter
\def\@author#1{\g@addto@macro\elsauthors{\normalsize%
    \def\baselinestretch{1}%
    \upshape\authorsep#1\unskip\textsuperscript{%
      \ifx\@fnmark\@empty\else\unskip\sep\@fnmark\let\sep=,\fi
      \ifx\@corref\@empty\else\unskip\sep\@corref\let\sep=,\fi
      }%
    \def\authorsep{\unskip,\space}%
    \global\let\@fnmark\@empty
    \global\let\@corref\@empty  
    \global\let\sep\@empty}%
    \@eadauthor={#1}
}
\makeatother

\makeatletter
\newcommand{\removelatexerror}{\let\@latex@error\@gobble}
\makeatother
\usepackage{xcolor}
\usepackage[ruled, vlined]{algorithm2e}

\SetCommentSty{mycommfont}

\newtheorem{theorem}{Theorem}[section]
\newtheorem{lemma}[theorem]{Lemma}
\newtheorem{proposition}[theorem]{Proposition}
\newtheorem{corollary}[theorem]{Corollary}
\newtheorem{definition}[theorem]{Definition}

\usepackage[title,titletoc]{appendix}

\let\svthefootnote\thefootnote
\newcommand\blankfootnote[1]{%
  \let\thefootnote\relax\footnotetext{#1}%
  \let\thefootnote\svthefootnote%
}

\modulolinenumbers[5]

\begin{document}


\title{A Matrix Representation of the Multiple Vehicle Routing Problem for Pickup and Delivery}





\author{Jinsun Liu$^{*,1}$, Hyongju Park$^{*,2}$, Matthew Johnson-Roberson$^{3}$, and Ram Vasudevan$^{4}$}

\blankfootnote{This work was supported by the Ford Motor Company.}

\blankfootnote{$^*$ These authors contributed equally to this work.}
\blankfootnote{$^1$ Robotics Institute, University of Michigan Ann Arbor, MI 48109 jinsunl@umich.edu}
\blankfootnote{$^2$ Mechanical Engineering, University of Michigan Ann Arbor, MI 48109 hjcpark@umich.edu}
\blankfootnote{$^3$ Naval Architecture and Marine Engineering, University of Michigan Ann Arbor, MI 48109 mattjr@umich.edu}
\blankfootnote{$^4$ Mechanical Engineering, University of Michigan Ann Arbor, MI 48109 ramv@umich.edu}

\maketitle






\begin{abstract}
This paper develops a computationally efficient algorithm for the Multiple Vehicle Pickup and Delivery Problem (MVPDP) with the objective of minimizing the tour cost incurred while completing the task of pickup and delivery of customers.
To this end, this paper constructs a novel $0$--$1$ Integer Quadratic Programming (IQP) problem to exactly solve the MVPDP. 
Compared to the state-of-the-art Mixed Integer Linear Programming (MILP) formulation of the problem, the one presented here requires fewer constraints and decision variables.
To ensure that this IQP formulation of the MVPDP can be solved in a computationally efficient manner, this paper devises a set of sufficient conditions to ensure convexity of this formulation when the integer variables are relaxed. 
In addition, this paper describes a transformation to map any non-convex IQP formulation of the MVPDP into an equivalent convex one. 
The superior computational efficacy of this convex IQP method when compared to the state-of-the-art MILP formulation is demonstrated through extensive simulated and real-world experiments. 
\end{abstract}




\section{Introduction}
The Multiple Vehicle Pickup and Delivery Problem (MVPDP) is a variant of the vehicle routing problem that deals with the assignment of customer demands to a set of vehicles and the scheduling of the sequence of customer pick-up and drop-off requests while minimizing travel cost \cite{cici2014assessing,gloss2016designing}.
Even solving the simpler single Vehicle Pickup and Delivery Problem (VPDP) is known to be challenging, since it is equivalent to the one-to-one version of the multi-commodities pickup-and-delivery traveling salesman problem, which belongs to the class of NP-Hard problems \cite{hernandez2009multi}. 
Nevertheless, given the potential utility of an algorithmic solution to this class of pickup and delivery problems, a variety of techniques have been proposed to tackle this problem.
Since a comprehensive survey of techniques to solve this problem developed prior to 2007 has been summarized in a pair of recent surveys \cite{cordeau2003dial,cordeau2007dial}, this paper focuses predominantly on describing more recent efforts.

A number of previous studies have focused primarily on obtaining sub-optimal solutions using heuristics.
For example, recent papers have applied population-based heuristics in the presence of time-window constraints (i.e. constraints on the time at which customers must be picked up and dropped off) and capacity constraints \cite{cherkesly2015population} or multi-stage hybrid heuristics \cite{hernandez2016hybrid}.
Others have explored variants of genetic algorithms \cite{ting2013selective}, particle swarm optimization \cite{goksal2013hybrid}, simulated annealing in the presence of time window constraints \cite{wang2015parallel,gansterer2017multi}, or hybrid heuristic algorithms \cite{ritzinger2016dynamic} to solve the MVPDP. 

Though the computational complexity of constructing the exact solution to the MVPDP problem can be demanding for large numbers of vehicles and demands, exact solution methods are critical for a variety of reasons \cite{cordeau2006branch}.
For instance, techniques that can generate exact solutions can be applied to enhance or validate the quality of solutions that are constructed using heuristic techniques.
One of the first approaches to propose an exact solution to the VPDP relied on dynamic programming \cite{psaraftis1980dynamic}.
The primary drawback of such dynamic programming based techniques is due to the \emph{curse of dimensionality}.
Though approximate dynamic programming approaches have been proposed to solve such VPDP problems when the problem size is large \cite{yin2016energy}, these techniques suffer from the same limitations as heuristic based approaches.
To improve the tractability of techniques that attempt to solve for the exact solution to such problems, some have proposed introducing additional constraints (e.g. time-window constraints) that reduce the number of states while performing dynamic programming \cite{psaraftis1983exact,desrosiers1986dynamic,mahmoudi2016finding}. 

Others have explored techniques to formulate the VPDP or MVPDP as an linear program with integer and real-valued decision variables \cite{cordeau2006branch}. 
These Mixed Integer Linear Programs (MILP) are solved using Branch-and-Bound \cite{colorni2001modeling,barbato2016polyhedral,kalantari1985algorithm}, Branch-and-Cut \cite{cordeau2006branch,cote2012branch,ruland1997pickup}, and Branch and Cut and Price \cite{ropke2009branch,ropke2007models}.
These techniques are usually able to solve problems with up to $18$ customers in several hours \cite{barbato2016polyhedral}; however, with the introduction of additional constraints (e.g. Last-In-First-Out constraints) these techniques can service up to $75$ customers in about an hour of computation time  \cite{cherkesly2016branch}.
With the introduction of time-window constraints, even larger instances of this problem are solvable \cite{dumas1991pickup}.

This paper considers a version of the MVPDP where a group of customers are located at some pick-up location in a bounded region. 
Each customer is also associated with a destination, which may be different from the pick-up location, where he/she must be dropped off.
One or more vehicles, each carrying only a limited number of customers, initially placed at some origin depot execute pickup and delivery of the customers, and return back to a designated destination depot after completion.
The objective of the MVPDP is to find a minimum length tour for all vehicles, where each tour starts and ends at the given depots, while servicing all customers along the tours.
As described earlier, the MILP formulations of the MVPDP without the introduction of additional assumptions are only able to address problems with fewer than $18$ customers in less than several hours using state-of-the-art commercial solvers.
This paper proposes a novel matrix-based Integer Quadratic Programming (IQP) formulation of the MVPDP, which in contrast to the MILP approach, is able to address problems with up to $50$ customers in the same amount of time using state-of-the-art commercial solvers.
This is accomplished without introducing constraints that may be hard to estimate \emph{a priori}.

Throughout the text, we make the following assumptions:
First, we assume that all demands are known \emph{a priori}. 
Though demand changes dynamically over time under traffic conditions  \cite{bertsimas1991stochastic,furuhata2013ridesharing}, the performance of optimal routing algorithms typically rely on the quality of the solution to the formulation of the problem where all demands are known \emph{a priori} \cite{berbeglia2007static}. 
In practice, the solution we propose in this paper can be used as the basis of the dynamic version of the MVPDP as has been proposed in various papers \cite{psaraftis1995dynamic,pillac2013review}.
Second, we assume that once a costumer has been picked up by a vehicle, he/she will stay on it until the vehicle reaches his/her destinations. 
This version of the problem, is sometimes called \emph{non-preemptive} \cite{hernandez2009multi}. 
Finally, each location in a tour is either a source or destination for one of the customers. 
In particular, we do not allow a point to simultaneously be the source of one customer and the destination of another customer. 
We do however allow the origin and the destination depot of each vehicle to be the same point and allow the cost incurred between a pair of points to be directional (e.g. different costs for going in one direction versus the opposite direction).

The contributions of this paper are three-fold: 
First, we propose a permutation matrix-representation of the MVPDP that reduces the complexity of the problem by reducing the number of constraints.
Due to this formulation of the problem, our method can be used to solve a broad class of MVPDP including those with multiple vehicle depots.
Second, our proposed IQP formulation can always be made convex (after relaxation of the integer variables).
Convex formulations of IQPs can be solved more rapidly than their non-convex counterparts \cite{buchheim2012effective}. 
Finally, on an extensive set of computational experiments we illustrate the effectiveness of our proposed formulation on various commercial solvers when compared to state-of-the-art MILP formulations on all tested commercial solvers. 

\textit{Organization:}
The remainder of the paper is organized as follows: Section \ref{sec2} introduces notation and briefly reviews the existing state-of-the-art exact solution method for the MVPDP. 
Section \ref{sec3} introduces the matrix representation of the problem. 
Section \ref{sec4} formulates the proposed IQP method to solve the MVPDP and illustrates how to ensure its convexity.
A suite of numerical simulation results are presented in Section \ref{sec5}, which is followed by the conclusion in Section \ref{sec6}.

\section{Preliminaries}
\label{sec2}
This section introduces the notations used throughout the paper, then reviews the characteristics of the most popular existing exact solution methods for the MVPDP.

\subsection{Notation and Background}

Consider a bounded region $Q \subseteq \mathbb{R}^2$ containing a group of $n$ customers. 
Also consider a group of $k$ vehicles that each have their own origin and destination depots. 
Each vehicle departs at its origin depot and arrives at its destination depot after performing an assigned set of pickup-and-delivery of  customers.
The vehicles must pick-up and deliver all $n$ customers.  
The locations of both origin and destination depot can coincide with one another. 
Each customer has their own origin and destination pair.
We also assume that each vehicle can carry up to $q\geq 1$ customers. 

For convenience, consider a set of vertices $V$ which correspond to the locations of the vehicle depots and customers' origins and the destinations.
We use both the term \emph{vertex} and \emph{node} interchangeably throughout the text. 
Let $O=\lbrace 1 \rbrace$ be a \emph{virtual node}\footnote{The virtual (dummy) node enables a matrix representation of our problem, and is addition does not affect the computational complexity of our formulation.}, $K_O= \lbrace 2,\dots,k+1 \rbrace$ be an ordered set of the origins for all $k$ vehicles, $K_D = \lbrace k+2, 2k+1 \rbrace$ be the ordered set of the destinations of all $k$ vehicles, $P = \lbrace 2k+2,\dots,2k+1+n \rbrace$ be an ordered set of $n$ pickup nodes, $D = \lbrace 2k+2+n,\dots,2k+2n+1 \rbrace $ be the set of ordered $n$ delivery nodes. 
Note that each origin and destination node are paired. 
For instance if a vehicle left the $2$nd node (origin depot) it must be returned to the associated $k+2$th node (destination depot). 
In a similar manner, each of the $n$ customers has an origin-destination pair. 
Finally, let $V = O \cup K_O \cup K_D \cup P \cup D$ be the set of all nodes. 
Let $v = \left| V \right| = 2(n+k)+1$ be the total number of vertices.



 Let $G(V,E,C)$ denote a \emph{weighted directed graph} without \emph{self-loops} using the previously defined vertex set $V$, arc (edge) set, $E \subseteq V \times V$, and a cost matrix $C$. 
Let $A$ be the $v \times v$ adjacency matrix for $G$ where the value of the $[i,j]$th entry (i.e., the entry in the $i$th row and the $j$ column) of the matrix $A$, denoted $a_{ij}$, represents the existence of the directed edge $(i,j)$. 
Thus if an edge exists between node $i$ and node $j$, then $a_{ij}=1$ and $a_{ij}=0$, otherwise. 
$C$ denotes the weight matrix for the graph where each entry of $C$, say $c_{ij}$ accounts for the weight incurred between the directed edge $(i,j)$. 
This weight assigned to each edge may correspond to the travel cost (e.g., fuel cost, trip time).
We assume that $C$ satisfies the triangle inequality\footnote{It is unknown if VPDP has an optimal solution tour if $C$ violates the triangle inequality \cite{kalantari1985algorithm}.}.
A \emph{walk} is an alternating sequence of vertices and edges which begins and ends at vertices, i.e., $v_0,e_1,v_1,e_2,\dots,v_{n-1},e_n,v_n$, where $v_i$ are vertices, and $e_i$ are the edge connecting $v_{i-1}$ and $v_i$.
A walk with no repeated edges is called a \emph{tour}, and a walk with no repeated vertices is called a \emph{path}. 
If $v_0=v_n$, then a tour is \emph{closed}. 
If the closed tour has no repeated vertices except for $v_0=v_n$, it is called a \emph{cycle}.
The length of a path from a graph is the number of edges that the path contains.
A cycle graph is a graph that itself is a cycle. 

We use italic bold font to denote vectors.
Let $\mathcal{P}$ be a class of permutation all matrices with size $v \times v$ whose cardinally is $v!$. 
For a given real symmetric matrix, $M$, let $M \succ 0$ if $M$ is positive-definite and $M \succeq 0$ if $M$ is positive-semi-definite. 
For a given $n\times 1$ vector $\bm{q}$, let $\textup{diag}(\bm{q})$ denote a diagonal matrix where its $i$th diagonal is the $i$th element of $\bm{q}$.

\subsection{Existing Exact Solution Methods for the MVPDP}

We briefly revisit the performance of the existing, state-of-the-art exact solution methods using an MILP formulation of the MVPDP \cite{cordeau2003dial,cordeau2007dial}.
These MILP are usually solved by Branch-and-Bound (B\&B) \cite{colorni2001modeling}, Branch-and-Cut (B\&C) \cite{cordeau2003dial,cordeau2006branch}, Branch-Price-and-Cut \cite{ropke2009branch} and forward Dynamic Programming (DP) \cite{psaraftis1980dynamic,desrosiers1986dynamic}. 
Typically the objective function in such formulations minimizes the total routing cost.
Several have explored modifications that simultaneously minimize customer wait-time by appending a linear term to the objective function \cite{cordeau2006branch,rahmani2016column}.
Given a set of $k$ vehicles, the MVPDP requires finding the $k$ minimum length tours for the vehicles each beginning at the appropriate origin depot and concluding at the appropriate destination depot while servicing all $n$ customers requests.
For a set of $k$ tours to be \emph{feasible}, each tour must satisfy a number of constraints:

\begin{itemize}[leftmargin=*]
\item \textit{Cycle Constraints}: each vehicle must visit a set of nodes exactly once. 
Two sets of constraints are required to enforce a tour to be a cycle.
\begin{itemize}
    \item \textit{Sub-tour Elimination Constraint}: for a tour to be a path it cannot contain any closed tours, called sub-tours, that do not visit every node in $G(V,E)$ \cite{desrochers1991improvements}. 
    To ensure there are no sub-tours in $G(V,E)$, one must consider all subsets of nodes $V$, and make sure that there is an edge leaving a node in the subset and entering a node \emph{not} in the subset. 
    To represent this constraint, one typically requires a combinatorial number of linear inequality constraints as a function of the number of nodes in a graph. 
    It was later discovered that the combinatorial number of constraints can be replaced with a quadratic number of constraints by introducing real-valued slack variables \cite{gendreau1997covering}.
    \item \textit{Degree Constraint}: each vertex must have exactly one incoming and one outgoing edge. 
    This constraint is represented as a set of linear equalities. 
\end{itemize}
\item \textit{Customer Precedence Constraint}: each customer's delivery cannot be preceded by the customer's pickup. 
This constraint is represented by a set of linear constraints; however it requires the addition of non-integer slack variables \cite{savelsbergh1995general}.
\item \textit{Association Constraint}: each pickup demand is paired with a unique delivery demand, and vice versa.  
This constraint is represented using a set of linear constraints \cite{ropke2009branch}.
\item \textit{Capacity Constraint}: each vehicle can carry only a certain number of customers at all times.
This again is represented by a set of linear inequalities.
\end{itemize}
Along with capacity constraint, sub-tour elimination constraint are sometimes referred to as \emph{generalized order constraints} \cite{ruland1997pickup,cordeau2007dial}. 
We summarize a few notable characteristics of the MILP formulation when compared to the formulation presented in this paper in Table \ref{table1}.
In contrast to the MILP-based formulation, the one presented in this paper does not require non-integer linear decision variables. 
In addition the formulation described in this paper requires far fewer constraints compared to the state-of-the-art MILP formulation.
\begin{table}[]
\small
\centering
\caption{Comparison of the characteristics of the constraints between the method proposed in this paper (IQP) and existing methods (MILP) \cite{cordeau2003dial}.}
\label{table1}
\begin{tabular}{@{}llll@{}}
\toprule
\emph{Type of constraints} & \emph{Program Type} &  \emph{Number of Constraints} \\ \midrule
\multirow{2}{*}{Degree, Association, Capacity} & MILP &  $O(n)$ \\
 & IQP &  $O(n)$ \\ \midrule
\multirow{2}{*}{\begin{tabular}[c]{@{}l@{}}Generalized Order\\ (Sub-tour Elimination, Precedence)\end{tabular}} & MILP &  $O(n^2)$ \\
 & IQP &  $O(n)$ \\ \bottomrule
\end{tabular}
\end{table}

\section{Permutation Matrix Representation}
\label{sec3}
This section introduces our novel matrix representation of the MVPDP. 
This representation extends a matrix-based formulation for the general quadratic assignment problem \cite{lawler1963quadratic,sahai2017continuous} by developing representations for MVPDP constraints (e.g., precedence, capacity, association)
which enable us to address multiple depots, vehicles, and customers. 
In addition, we develop a technique to transform non-convex quadratic assignment problems into convex ones when the integer variables are relaxed. 
To the best of our knowledge, the method proposed in this paper is the first matrix representation of the problem that can be used to solve MVPDP. 
This formulation of the problem allows one to represent the constraints associated with the pickup-and-delivery problem without requiring the introduction of additional non-integer slack variables, which may unnecessarily impede computational efficiency.

Consider a sequence $(l_1,l_2,\dots,l_{v-1},l_v)$ where $l_i \in V$ for each $i \in \{1,\ldots,v\}$ and $l_i \neq l_j$ for all $i,j \in \{1,\ldots,v\}$ where $i \neq j$.
For a given set of vertices $V$, denote an arbitrary closed path---that begins at $l_1$ and sequentially visits $l_2,l_3,\dots,l_{v}$ exactly once, and returns back to $l_1$---by $\bm{\sigma}$ which is defined as:
\begin{equation}
\bm{\sigma} = [l_2,l_3,\dots,l_{v-1},l_v,l_1]^{\top},
\label{eq:sigma}
\end{equation}
where the initial path $(l_1,l_2)$ is \emph{implicit} in this vector representation.
Consider a tour denoted $\bm{\sigma}^0$ whose elements are drawn in order first from $K_0$, then $K_D$, then $P$, then $D$, and finally from $O$. 
We refer to $\bm{\sigma}^0$ as the \emph{reference tour}.
We refer to any other possible tour as a \emph{permuted tour}.

Denote the adjacency matrix and the graph associated with the reference tour by $A^0$ and $G^0$, respectively. 
It can be verified that $A^0$ is a $v \times v$ \emph{upper $1$-shift matrix}, whose $[i,j]$th entry is:
\begin{equation}
a_{ij}^0 = 
\begin{cases}
1, & \textup{if }i = j-1+v\left\lfloor \frac{i+1}{v} \right\rfloor, \\
0, & \textup{otherwise},
\end{cases}
\end{equation}
An example of a graph $G^0$ and its associated adjacency matrix $A^0$  for a reference tour, $\bm{\sigma}^0$ is depicted in Fig. \ref{fig:fig1_1}. 
for ease of presentation, in the remainder of the paper, we use the index $1,\ldots,v$ to refer to the nodes from $O$, then $K_O$, then $K_D$, then $P$, and finally $D$ in order.
For instance, in Fig. \ref{fig:fig1_1}, $\lbrace 6,7,8 \rbrace$ are pickup nodes, $\lbrace 9,10,11 \rbrace $ are delivery nodes, and $1$ is a virtual (dummy) node. 
It can be verified that $\bm{\sigma}^0$ is infeasible since each vehicle's tour is completed before the pickup and delivery of any customers.
Fig. \ref{fig:fig1_2} shows a feasible tour, where the 1st vehicle begins with node $2$ and travels from node $6$ to node $9$ (Origin-Destination pair for customer 1) and arrives at its depot node $4$, while the second vehicle begins with node $3$ and travels from node $8$ to node $11$ (O-D pair for customer 3), then from node $7$ to node $10$ (Origin-Destination pair for customer 2), before reaching depot node $5$. 
As can be seen, the valid tour is a cycle that begins and ends at the virtual node $1$ and visits all other nodes $\lbrace 2,3,\dots,12 \rbrace$ exactly once.
\begin{figure*}
    \centering
    \begin{subfigure}[b]{0.5\textwidth}
        \centering
        \includegraphics[width=2.6in]{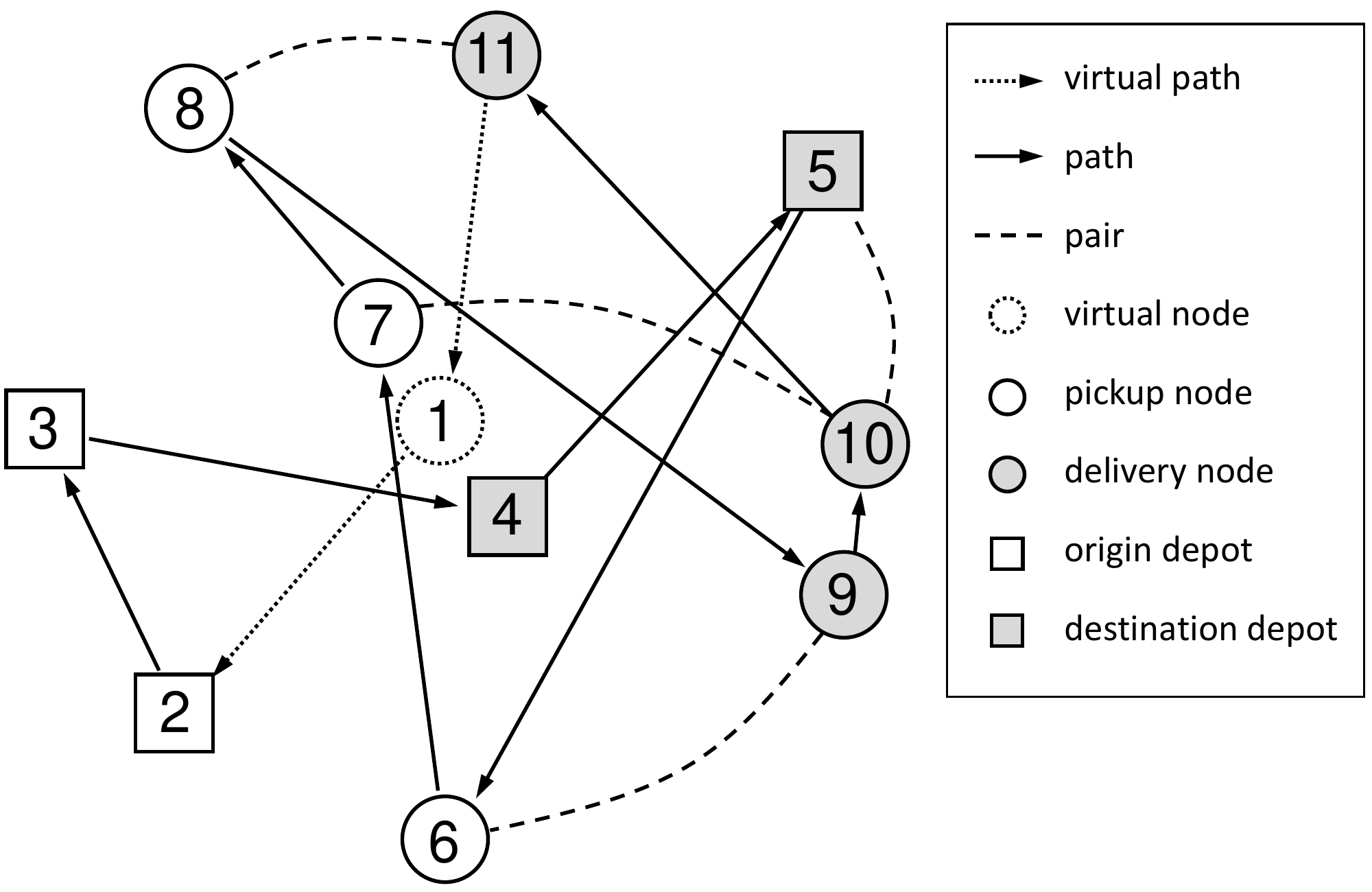}
        \caption{}
    \end{subfigure}%
    ~ 
    \begin{subfigure}[b]{0.5\textwidth}
        \centering
        \includegraphics[width=2.4in]{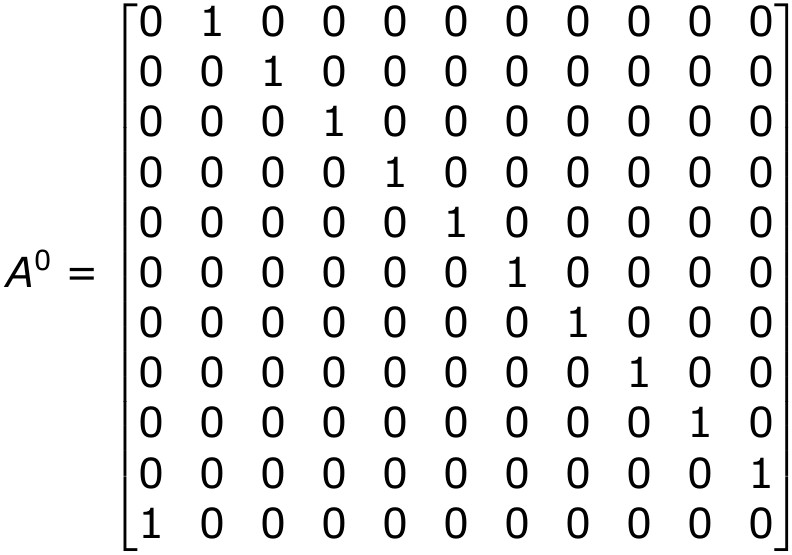}
        \caption{}
    \end{subfigure}
    \caption{An illustration of the reference cycle graph $G^0$ (a) and its associated adjacency matrix $A^0$ (b) for a reference tour, $\bm{\sigma}^0$, that is infeasible.}
    \label{fig:fig1_1}
\end{figure*}
\begin{figure*}
    \centering
    \begin{subfigure}[b]{0.5\textwidth}
        \centering
        \includegraphics[width=2.6in]{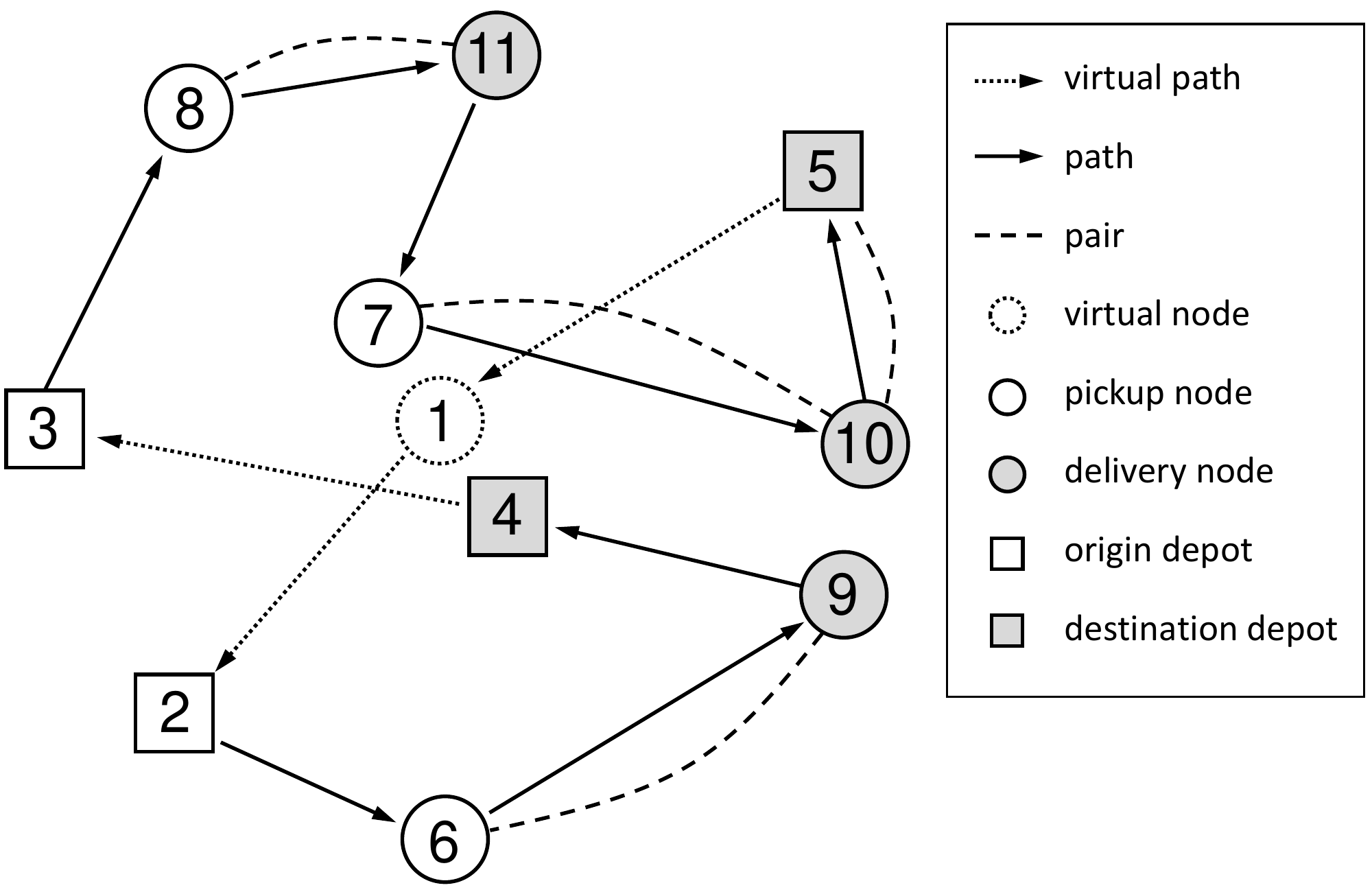}
        \caption{}
    \end{subfigure}%
    \begin{subfigure}[b]{0.5\textwidth}
        \centering
        \includegraphics[width=2.4in]{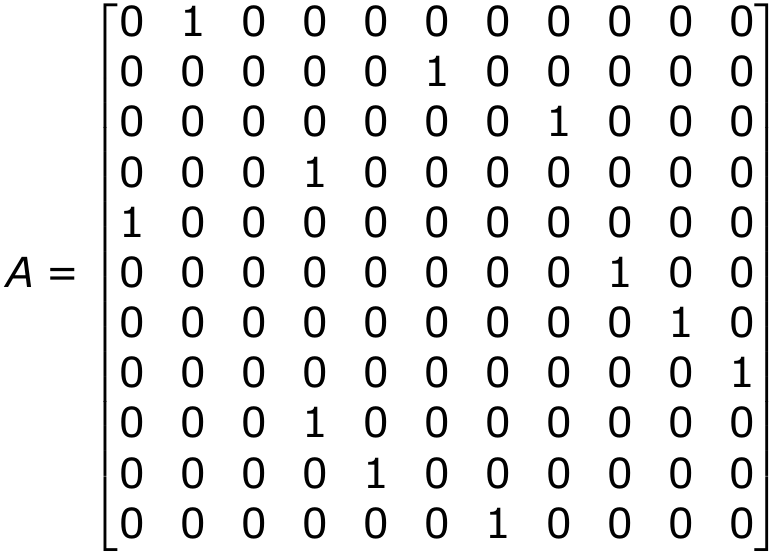}
        \caption{}
    \end{subfigure}%
    \caption{An illustration of the cycle graph $G$ (a) and its associated adjacency matrix $A$ (b) for a feasible tour.}
    \label{fig:fig1_2}
\end{figure*}

Though the tours depicted in Figs. \ref{fig:fig1_1} and \ref{fig:fig1_2} seem unrelated, we next discuss the relationship between $A^0$ and an arbitrary adjacency matrix $A$ corresponding to a tour defined over an identical set of vertices using permutation matrices.
\begin{proposition}
	Let $G^0$ is a cycle graph with adjacency matrix $A^0$.
	Let $G$ be a cycle graph defined over the same vertex set with adjacency matrix $A$. 
	Then there is a matrix $X \in {\cal P}$ such that:
	\begin{equation}
	A = X^{\top}A^0 X.
	\end{equation}
	\label{prop1}
\end{proposition}
\begin{proof}[Proof of Proposition \ref{prop1}]
 By Proposition \ref{isograph}, any two cycles graphs defined over the same vertex set are \emph{isomorphic}. 
 Thus, $G^0$, and $G$, are isomorphic. 
 $G^0$ is isomorphic to $G$ if and only if there exists a permutation matrix $X$ such that $A = X^{\top}A^0 X$ \cite[p.520-521]{turner1968generalized}.
 \end{proof}


\section{Convex Integer Quadratic Programming Problem Formulation}
\label{sec4}

This section describes our Integer Quadratic Programming (IQP) formulation to solve the MVPDP.
We begin by first formulating each of the constraints required to describe the MVPDP problem using permutation matrices.
We subsequently describe how to represent the cost function using permutation matrices and a technique to render any non-convex cost function as an equivalent convex cost function.
The section concludes by describing the IQP formulation of the MVPDP.
Note again, that for ease of presentation, in the remainder of the paper, we use the index $1,\ldots,v$ to refer to the nodes from $O$, then $K_O$, then $K_D$, then $P$, and finally $D$ in order.

\subsection{MVPDP Linear Constraint Formulation}

To simplify the formulation of the constraints, we introduce additional notation:
Let $\bm{e}_i$ be $i$th column of $v\times v $ identity matrix, i.e., $i$th standard basis in $\mathbb{R}^{v}$, $T$ be a $v \times v$ upper triangular matrix where its $i$th row is:
$
[\underbrace{0,\dots,0}_{i-1 \textup{ times}},\underbrace{1,\dots,1}_{v-i+1 \textup{ times}}]
$ for each $i \in V$, 
and let $p$ be the $v \times 1$ column vector defined by $
\bm{p}:= [\underbrace{0,\dots,0}_{2k+1 \textup{ times}},\,\underbrace{p_1,\dots,p_n}_{n \textup{ times}},\underbrace{-p_1,\dots,-p_n}_{n \textup{ times}}]^{\top},$ where $p_i$ is the number of customers that are associated with customer demand $i$.
Finally let $\bm{n}:= [1,\,2,\,\dots,\,v]^{\top}$.


\textit{Path Constraint}: 
First, note that $X$ is a permutation matrix if and only if each of its row and column sums is one and all the elements are either $0$ or $1$.
To enforce this constraint, each element of $X$ must be either $0$ or $1$ and 
\begin{equation}
X \bm{1} = \bm{1} \quad \text{and} \quad X^{\top} \bm{1} = \bm{1}.
\end{equation}
Note that since our optimization problem uses permutation matrices as its decision variables, it does not require sub-tour elimination constraints. 
This is because for any given permutation matrix $X$, and a reference cycle, $\bm{\sigma}^0$, a tour generated by the permutation matrix $X$ represented as $\bm{\sigma} = X \bm{\sigma}^0 $ is necessarily a \emph{single} cycle, which does not contain any sub-tours.

\textit{Vehicle Precedence}: 
The following proposition summarizes a set of constraints that enforces the vehicle precedence constraint:
\begin{proposition}[Vehicle Precedence]
\label{prop41}
Suppose $X \in {\cal P}$ satisfies the following inequalities:
\begin{align}
&\bm{n}^{\top}X(\bm{e}_{1+i}-\bm{e}_{1+i+k}) \leq -\bm{1},\,\;\textup{for }i = 1,\dots,k \label{eq1},\\
&\bm{n}^{\top}X(\bm{e}_{1+k+i}-\bm{e}_{2+i}) = -\bm{1},\,\;\textup{for }i = 1,\dots,k-1 \label{eq2},
\end{align}
then $X\bm{\sigma}^0$ is a tour wherein each vehicle's origin is preceded by its destination and each vehicle's destination immediately precedes the origin of the following vehicle. 
\end{proposition}
The proof of the proposition follows by observation.
Fig. \ref{fig:fig12} illustrates a tour that satisfies the constraints specified by \eqref{eq1} and \eqref{eq2}.
The dashed line from Fig. \ref{fig:fig12} indicates any possible length-m path between the $i$th vehicle's origin node and its destination node, where $m \geq 1$, and the solid line indicates that the $(i+1)$th vehicle's origin depot immediately follows the $i$th vehicle's destination node. 
Notice that by enforcing \eqref{eq1} and \eqref{eq2} and ensuring that all other nodes are placed and ordered between one of the vehicle pickup and delivery depot node pairs, a single cycle that begins and ends at the virtual node can be constructed.
\begin{figure}[h]
\centering
    \begin{subfigure}[b]{0.5\textwidth}
	\centering
	\includegraphics[width=2.7in]{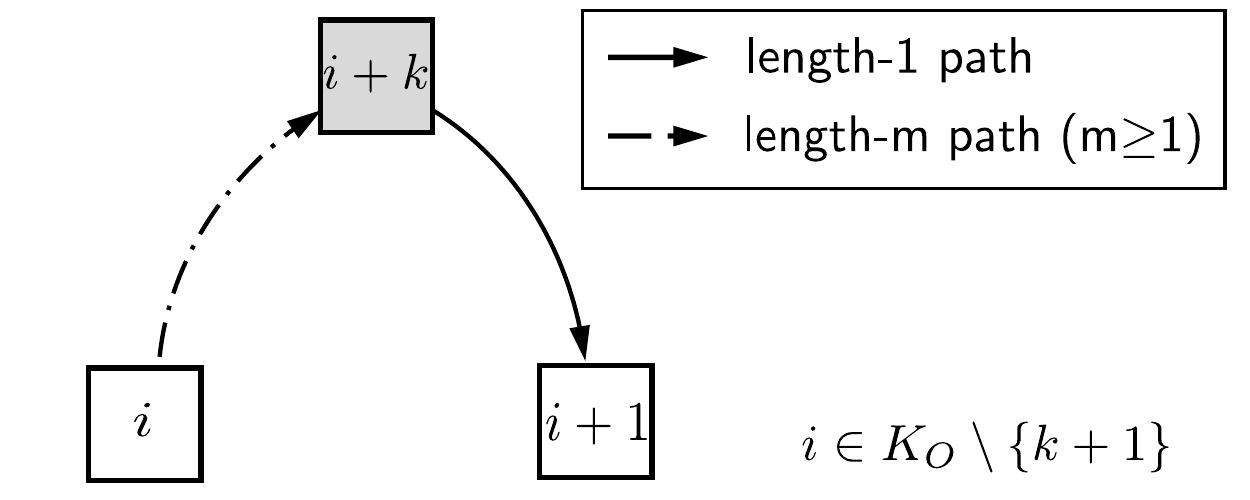}
	\caption{}
	\label{fig:fig12}
    \end{subfigure}%
    \begin{subfigure}[b]{0.5\textwidth}
	\centering
	\includegraphics[width=2.7in]{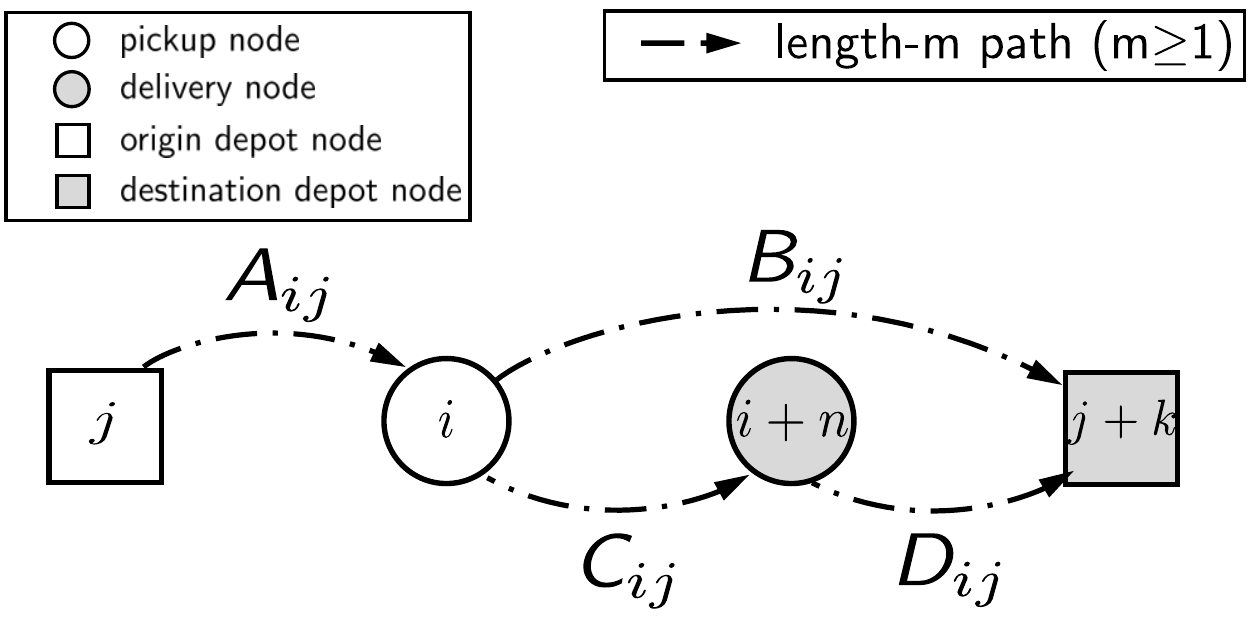}
	\caption{}
	\label{fig:figstupidmistake}
    \end{subfigure}%
    \label{fig:fig123}
    \caption{Illustrative explanations of the equality and inequality constraints for (a) the vehicle precedence, and (b) customer precedence.}
\end{figure}

\textit{Association and Customer Precedence}
The following proposition proposition summarizes a set of constraints related to vehicle-customer association and customer precedence:
\begin{proposition}
Suppose $X \in {\cal P}$ satisfies \eqref{eq1}-\eqref{eq2} and that whenever:
\begin{equation}
\underbrace{(\bm{n}^{\top}X(\bm{e}_{1+j}-\bm{e}_{2k+1+i}) \leq -\bm{1})}_{\textsf{A}_{ij}} \wedge \underbrace{(\bm{n}^{\top}X(\bm{e}_{2k+1+i}-\bm{e}_{1+k+j}) \leq -\bm{1}}_{\textsf{B}_{ij}}) \label{eq3a}
\end{equation}
is satisfied that the following inequality is also satisfied:
\begin{equation}
\underbrace{(\bm{n}^{\top}X(\bm{e}_{2k+1+i}-\bm{e}_{2k+1+i+n}) \leq -\bm{1})}_{\textsf{C}_{ij}} \wedge \underbrace{(\bm{n}^{\top}X(\bm{e}_{2k+1+i+n}-\bm{e}_{1+k+j}) \leq -\bm{1}}_{\textsf{D}_{ij}}\label{eq3b}
\end{equation}
for each $j \in \lbrace 1,\dots,k \rbrace$, $i \in \lbrace 1,\dots,n \rbrace$ where `$\wedge$' means `and.'
Then $X \bm{\sigma}^0$ is a tour in which each customer's delivery is preceded by their pickup and carried out by the car that picked them up.
\label{prop51}
\end{proposition}
\begin{proof}
As can be seen from Fig \ref{fig:figstupidmistake}, if all the vehicle precedent constraints, \eqref{eq1}-\eqref{eq2}, are satisfied, then for each $i,j$, both the inequalities $\textsf{\textit{A}}_{ij}$ and ${\textit{B}}_{ij}$ hold, if and only if $i$th customer was pickup by $j$th vehicle, which implies that $i$th customer will be delivered by the $j$th vehicle, which is expressed by the two inequalities $\textsf{\textit{C}}_{ij}$ and $\textsf{\textit{D}}_{ij}$.
\end{proof}

To enforce a conditional constraint such as \eqref{eq3a} implies \eqref{eq3b}, we introduce an auxiliary binary variable and use the Big-M method \cite{winston2003introduction}.
For instance, consider three binary variables $x_1,x_2, y \in \lbrace 0,1\rbrace$. 
The conditional statement $a_1x_1 \geq b_1 \Rightarrow a_2x_2 \geq b_2$, where $\Rightarrow$ means `implies,' can be represented as: 
\begin{align}
    & a_1x_1 < b_1 + My, \label{eqfst} \\
    & a_2x_2 \geq b_2 - M(1-y),\nonumber \\
    & y_1 \in \lbrace 0,1\rbrace \nonumber
\end{align}
for a large constant $M$. 
Notice that the strict inequality term  \eqref{eqfst} can be replaced with $a_1x_1 \leq b_1-1 + My$, if $a_1$, $b_1$ are both integers.

%

\textit{Capacity of each vehicle}: 
The following proposition proposition summarizes a constraint to enforce vehicle capacity:
\begin{proposition}
Suppose $X \in {\cal P}$ satisfies Proposition \ref{prop51} and the following inequality:
\begin{equation}
T X\bm{p}\leq q1.
\end{equation}
Then $X\bm{\sigma}^0$ is a tour that satisfies the vehicle capacity constraint at all times. 
\label{prop31}
\end{proposition}
The proof of Proposition \ref{prop31} follows by observation.
We note that satisfying the constraints in Proposition \ref{prop41} and \ref{prop51} ensures that at every vehicle depot, the occupancy of the vehicle is reset to $0$.
Then, $TX\bm{p}$ is a vector whose $i$th component corresponds to the number of customers in a vehicle at $i$th position of the tour, which should be at most $q$ for all $i \in V$.

\subsection{MVPDP Convex Quadratic Cost Function}

For a given cycle graph $G$ and its associated adjacency matrix $A$ and its cost matrix $C$, the total tour cost is computed by
\begin{equation}
\textup{trace}(C^{\top}A) = \bm{1}^{\top}C \circ A \bm{1},
\label{eqcost}
\end{equation}
where $\circ$ is a \emph{Hadamard} product operator which is used for element-wise product of matrices, and $\bm{1}$ is a $v\times 1$ vector with ones.
Since our decision variable in our optimization problem is a permutation matrix, the cost function as a result of Proposition \ref{prop1} for some permutation matrix $X$ becomes:
\begin{equation}
    \textup{trace}(C^{\top}A) = \textup{trace}(C^{\top}X^T A^0 X ),
\end{equation}
which may be a non-convex function of $X$. 
A recent survey has shown that optimization solvers for a variety of integer nonlinear programs (including IQPs) are much easier to solve if they are convex when the integer variables are relaxed \cite{burer2012non}. 
In fact, the computation time for these integer programs that are convex when the integer variables are relaxed has been shown to be much faster than their non-convex counterparts in various studies \cite{fletcher1998numerical,bonami2008algorithmic,bonami2012algorithms,takapoui2017simple}. 
 
The method we use to \emph{convexify} the original non-convex program is by adding a constant term to the original cost function \eqref{eqcost} and formulate a convex program and subtracting the term afterwords.
First, we note by the following proposition that the value of \eqref{eqcost} does not depend on the diagonal entries of $C$.
\begin{proposition}
For a given directed graph, consider an adjacency matrix without a self-loop, $A$, and a square matrix $C$.
For each real diagonal matrix $D$ of the same size as $C$, the following is true:
\begin{equation}
	\textup{trace}((C+D)^{\top}A)  = 	\textup{trace}(C^{\top}A).
\end{equation}
	\label{prop0}
\end{proposition}
\begin{proof}
The proof is immediate from the fact that for the given graph without self-loop, its adjacency matrix, $A$ has only zeros on its diagonal so:
\begin{equation}
\textup{trace}((C+D)^{\top}A) = \textup{trace}(C^{\top}A) + \cancelto{0}{\textup{trace}(D^{\top}A)}.
\end{equation}
\end{proof}
Thus, we may construct the cost matrix $C$ with arbitrary diagonal entries without affecting the function value of \eqref{eqcost}.
Next, we reformulate the previous matrix representation of the cost, as a quadratic function in terms of $\bm{x}:= \text{vec}(X)$ which is a vectorization of $X$. 
Consider an adjacency matrix $A$ of an arbitrary cycle graph $G$. 
The cost incurred while completing a single cycle through $G$ can be represented as:
\begin{align}
    \textup{trace}(C^{\top}A)&=\textup{trace}(C^{\top}X^{\top}A^0X) && \textup{by Proposition }   \ref{prop1} \\
    &=\textup{trace}(C^{\top}X^{\top}(A^0+\gamma I)X) - \gamma \textup{trace}(C^{\top} X^{\top}X) && \textup{by the linearity of the trace operator}  \\
    &=\textup{trace}(C^{\top}X^{\top}(A^0+\gamma I)X) - \gamma \textup{trace}(C) && \textup{by the properties of permutation matrices}  \label{lasteq}
\end{align}
where $\gamma \geq 0$. 
By using the last form of the tour cost \eqref{lasteq}, we define a cost functional, namely $L$, 
\begin{equation}
L: (X,\gamma) \mapsto \textup{trace}(C^{\top} X^{\top} (A^0+\gamma I) X)-\gamma\textup{trace}(C),
\end{equation}
where $\gamma \geq 0$.
Now let
\begin{equation}
Q := \frac{1}{2}H_{\bm{x}}^{L},
\end{equation}
where we used $H_{\bm{x}}^{L}$ to denoted the \emph{Hessian} of $L$ with respect to $\bm{x}$.
Then, one can show:
\begin{equation}
Q= 
\underbrace{\begin{bmatrix}
	\gamma\overline{C} & C/2      & 0  & \cdots & 0 & C/2 \\
	C^{\top}/2   & \gamma\overline{C} &  C/2 & 0 & \cdots & 0\\
	0 & C^{\top}/2 & \ddots & \ddots & \ddots & \vdots \\
	\vdots& \ddots & \ddots & \ddots & C/2 & 0 \\
	0   & \cdots      & 0  &  C^{\top}/2   &\gamma\overline{C} & C/2  \\
	C^{\top}/2 & 0      & \cdots & 0 & C^{\top}/2 & \gamma\overline{C}
	\end{bmatrix}}_{v \textup{ blocks}}
	\label{qmat}
\end{equation}
where $\overline{C} = (C+C^{\top})/2$, and $0$ is the $v\times v$ matrix where all entries are zeroes.
Note $Q$ is symmetric.

Consider a $v^2 \times v^2$ real diagonal matrix $D = \textup{diag}(\bm{d})$ where $\bm{d}=[d_1,\dots,d_{v^2}]^{\top}$.
Using Proposition \ref{prop1} and the property of trace operator, the quadratic form of the cost can be defined as follows:
\begin{align}
   L(X,\gamma) &=\frac{1}{2}\bm{x}^{\top}H_{\bm{x}}^L \bm{x} - \gamma \textup{trace}(C) && \textup{since } \textup{trace}(C^{\top} X^{\top} (A^0+\gamma I) X)=\frac{1}{2}\bm{x}^{\top}H_{\bm{x}}^L\bm{x}  \\
    &=\bm{x}^{\top}(\underbrace{Q+D}_{\widetilde{Q}:=Q+D}) \bm{x}- \gamma \textup{trace}(C) - \bm{x}^{\top}D\bm{x} && \textup{since }Q=\frac{1}{2}H_{\bm{x}}^L  \\
    &=\bm{x}^{\top}\widetilde{Q}\bm{x} - (\gamma \textup{trace}(C) +  \bm{x}^{\top}D\bm{x})
    \\
    &=\bm{x}^{\top}\widetilde{Q}\bm{x} - \bm{x}^{\top}\bm{d} - \gamma \textup{trace}(C)   &&\textup{since }\bm{x}^{\top}D\bm{x} = \bm{x}^{\top}\bm{d}\textup{ for any }\bm{x} = \textup{vec}(X) \text{ with } X \in {\cal P}
    \label{eqmain}
\end{align}

The following lemma, which is proven in Appendix B, illustrates how to select $\gamma$ and $D$ to ensure that $\widetilde{Q}:= Q + D$ is positive semi-definite:



\begin{lemma}
	For a given $C \in \mathbb{R}^{v\times v}$, if
\begin{align*}
	& d_i \geq \min
	\left\{
\sum_{j=1}^v c_{ji},\,\sum_{j=1}^vc_{ij}
	\right\},\,\,&& \textup{ for }i =v+1,\dots,v^2-v \\
& d_i \geq \frac{\sum_{j=1}^v(c_{ji}+c_{ij})}{2} && \textup{ for }i =1,v^2-v+1,\dots,v^2
\end{align*}
then \eqref{eqmain} is convex with $\gamma = 0$.
Furthermore, if $C$ is symmetric and positive semidefinite, and $\gamma > 1$, then \eqref{eqmain} is convex with $D=0$.
	\label{mainlem}
\end{lemma}
The first part of the lemma is more general than since it does not restrict $C$ to be a positive semidefinite matrix. 
While the second part of the lemma requires $C$ to be a symmetric, positive semidefinite matrix, this is true in many real-world applications of the MVPDP problem since the weight of each arc corresponds to the tour distance, which is usually symmetric. 
By adjusting the weights of the self-loops, $C$ can be made positive semidefinite without changing the overall cost due to Proposition \ref{prop0}. 
While it might be tempting to do so, the second part of the lemma cannot be proven trivially or using the same techniques that were used in the first part of the proof. 

\subsection{MVPDP IQP Formulation}

Using the set of valid constraints found from the previous sections, along with our cost function, our problem can be formulated as a convex IQP with linear constraints as follows:

\begin{align}
& \underset{\bm{x} = \textup{vec}(X)}{\textup{min}}
& & \bm{x}^{\top}\widetilde{Q} \bm{x}  - \bm{d}^{\top}\bm{x} - \gamma \textup{trace}(C)  \nonumber \\
& \text{s.t.}
& &X \bm{1} = \bm{1},\, X^{\top} \bm{1} = \bm{1},
\label{con1}
\\
&&&x_i \in \lbrace 0,\,1 \rbrace, \textup{ for } i = 1,\dots,v^2 
\label{con2}
\\
&&&\bm{n}^{\top}X(\bm{e}_{1+i}-\bm{e}_{1+i+k}) \leq -\bm{1},\,\;\textup{for }i = 1,\dots,,k
\label{con3}
\\
&&&\bm{n}^{\top}X(\bm{e}_{1+k+i}-\bm{e}_{2+i}) = -\bm{1},\,\;\textup{for }i = 1,\dots,k-1,
\label{con4}
\\
&&&\left((\bm{n}^{\top}X(\bm{e}_{1+j}-\bm{e}_{2k+1+i}) \leq -\bm{1}) \wedge (\bm{n}^{\top}X(\bm{e}_{2k+1+i}-\bm{e}_{1+k+j}) \leq -\bm{1})\right) \Rightarrow \nonumber \\ 
&&& \left((\bm{n}^{\top}X(\bm{e}_{2k+1+i}-\bm{e}_{2k+1+i+n}) \leq -\bm{1}) \wedge (\bm{n}^{\top}X(\bm{e}_{2k+1+i+n}-\bm{e}_{1+k+j}) \leq -\bm{1})\right) \nonumber \\
&&& \phantom{ stupid text to space things out mr } \textup{for }j \in \lbrace 1,\dots,k \rbrace,\,i \in \lbrace 1,\dots,n \rbrace,\label{con5}\\
&&& T X\bm{p} \leq q\bm{1}  \label{con6}
\end{align}
where \eqref{con1} is the degree constraint, \eqref{con2} is the binary variable constraint, \eqref{con3}-\eqref{con4} are the vehicle precedence constraints, 
\eqref{con5} is the customer precedence and association constraint (which is implemented using the Big-M method), and \eqref{con5} is the capacity constraint.
Importantly, as a result of Lemma \ref{mainlem}, the cost function of the IQP regardless of the specific permutation matrix is convex. 
%
%

\section{Numerical Simulations}
\label{sec5}
This section evaluates the performance of our method using a set of extensive computational experiments, which are all implemented in MATLAB (R2017a)\footnote{
The MATLAB scripts and all the test data that were used for the simulation can be found online at ``\url{https://github.com/hyongju/mvpdp}.''
}. 
First, we demonstrate the computational benefits of our formation when compared to the state-of-the-art formulation for problems with $8-50$ customer demands that are each serviced by a single vehicle.
To rigorously evaluate the performance of our IQP based formulation when compared to the state-of-the-art MILP formulation we use a variety of different types of commerical optimization solvers. 
Next, we show that our method can solve problems with multiple depots and vehicles to optimality.
Finally, we show the applicability of our method to solve real-world ride-sharing problems by utilizing real-world demand data.


\subsection{Evaluation While Servicing Between $8-50$ Demands Using a Single Vehicle}

We begin by evaluating the effectiveness of our method when compared to the state-of-the-art MILP formulation while solving problems with a single vehicle that is servicing demands for $8-50$ customers.
To solve each problem, we conduct simulations on our server computer (Intel Xeon E7-8867 v4 at 2.40GHz with 1023Gb memory) on MATLAB using five state-of-the-art MIP solvers, namely, CPLEX (12.7.1), GUROBI (7.5.1), MOSEK (8.0.0), BARON (17), XPRESS (18.1), where each value between the parenthesis indicate the solver's version number. 
For each solver, we tested $9$ possible customer demand values, $N:=\lbrace 
8,\,10,\,12,\,15,\,20,\,25,\,30,\,40,\,50 \rbrace$.
For each $N$, we generated $24$ random sample instances and generated an initial guess at random that was used for both the IQP and MILP solvers.
The corresponding symmetric cost matrix is generated from Euclidean distances between each pair of nodes and used by both the MILP and IQP methods.
For all test instances, $\gamma$ was set to $1.01$ to ensure that the IQP is convex (notice that each cost matrix $C$ is symmetric and positive semidefinite and this selection of $\gamma$ was sufficient to ensure the convexity of the IQP as a result of Lemma \ref{mainlem}).

The (relative) MIP gap is defined as the difference between best known (incumbent) solution and the best bound divided by the best lower bound found by a relaxation of the integer program.
Since it is impossible to solve for an exact result for problems of this size, we terminate each solving process either when the solving time exceeds $2$ hours, or the relative MIP gap is below a threshold value of $3.5\%$. 
We call a solution \textit{feasible} if it satisfies all the constraints.
We call a problem instance a \textit{success} if the MIP relative gap is no greater than $3.5\%$ and it is generated in $2$ hours.
We call a problem instance a \textit{partial success} if a feasible solution is found, but the MIP relative gap fails to reach $3.5\%$ within $2$ hours.
We call a problem instance a \textit{failure} if no feasible solution is found within $2$ hours.
Table \ref{suc_rate} illustrates the success rate and partial success rate for all problem instances with different solvers. 
Since the relative gap was chosen identically for both IQP and MILP when the solvers were successful the costs that were computed were nearly identical.
Notice that though GUROBI and MOSEK can actually find a feasible solution within 2 hours for up to 25 customers in the IQP formulation, they fail to achieve the $3.5\%$ gap.
This means except for BARON, all other solvers could obtain a feasible solution of circumstances with more customers for IQP than MILP.
Moreover XPRESS and CPLEX's IQP solvers are consistently able to solve almost all problem instances successfully whereas every MILP formulation is unable to be solved when the customer demand is greater than 12. 
Fig. \ref{fig:fig1} compares the average solving time/success rate to achieve within $3.5\%$ relative optimality gap between IQP and MILP using all five solvers, and Fig. \ref{fig:fig1CPLEX} shows the results for the CPLEX solver. 
\begin{table}[]
\scriptsize
\centering
\caption{Success rate and partial success rate (in parenthesis) for $n$ customer demands using a single vehicle and a maximum of $2$ hours of computation time (both rates are in percentage).}
\label{suc_rate}
\begin{tabular}{@{}l|l|l|l|l|l|l|l|l|ll@{}}
\toprule
\multirow{2}{*}{$n$} & \multicolumn{2}{c|}{CPLEX}                           & \multicolumn{2}{c|}{GUROBI}                          & \multicolumn{2}{c|}{MOSEK}                           & \multicolumn{2}{c|}{BARON}                           & \multicolumn{2}{c}{XPRESS}                                 \\ \cmidrule(l){2-11} 
                   & \multicolumn{1}{c|}{MILP} & \multicolumn{1}{c|}{IQP} & \multicolumn{1}{c|}{MILP} & \multicolumn{1}{c|}{IQP} & \multicolumn{1}{c|}{MILP} & \multicolumn{1}{c|}{IQP} & \multicolumn{1}{c|}{MILP} & \multicolumn{1}{c|}{IQP} & \multicolumn{1}{c|}{MILP}        & \multicolumn{1}{c}{IQP} \\ \midrule
8                  & 100 (0)               & 100 (0)              & 100 (0)               & 95.83 (4.17)         & 100 (0)               & 4.17 (95.83)         & 100 (0)               & 100 (0)              & \multicolumn{1}{l|}{100 (0)} & 100 (0)             \\
10                 & 100 (0)               & 100 (0)              & 100 (0)               & 0 (100)              & 100 (0)               & 0 (100)              & 95.83 (0)             & 100 (0)              & \multicolumn{1}{l|}{100 (0)} & 100 (0)             \\
12                 & 0 (0)                 & 100 (0)              & 95.83 (0)             & 0 (100)              & 91.67 (0)             & 0 (100)              & 41.67 (0)             & 95.83 (0)            & \multicolumn{1}{l|}{100 (0)} & 100 (0)             \\
15                 & 0 (0)                 & 100 (0)              & 0 (0)                 & 0 (100)              & 0 (0)                 & 0 (100)              & 0 (0)                 & 0 (0)                & \multicolumn{1}{l|}{0 (0)}   & 100 (0)             \\
20                 & 0 (0)                 & 100 (0)              & 0 (0)                 & 0 (95.83)            & 0 (0)                 & 0 (79.17)            & 0 (0)                 & 0 (0)                & \multicolumn{1}{l|}{0 (0)}   & 100 (0)             \\
25                 & 0 (0)                 & 100 (0)              & 0 (0)                 & 0 (45.83)            & 0 (0)                 & 0 (20.83)            & 0 (0)                 & 0 (0)                & \multicolumn{1}{l|}{0 (0)}   & 100 (0)             \\
30                 & 0 (0)                 & 100 (0)              & 0 (0)                 & 0 (0)                & 0 (0)                 & 0 (0)                & 0 (0)                 & 0 (0)                & \multicolumn{1}{l|}{0 (0)}   & 95.83 (0)           \\
40                 & 0 (0)                 & 100 (0)              & 0 (0)                 & 0 (0)                & 0 (0)                 & 0 (0)                & 0 (0)                 & 0 (0)                & \multicolumn{1}{l|}{0 (0)}   & 0 (0)               \\
50                 & 0 (0)                 & 97.92 (0)            & 0 (0)                 & 0 (0)                & 0 (0)                 & 0 (0)                & 0 (0)                 & 0 (0)                & \multicolumn{1}{l|}{0 (0)}   & 0 (0)               \\ \bottomrule
\end{tabular}
\end{table}

\begin{figure}[h]
	\centering
	\includegraphics[width=2.6in]{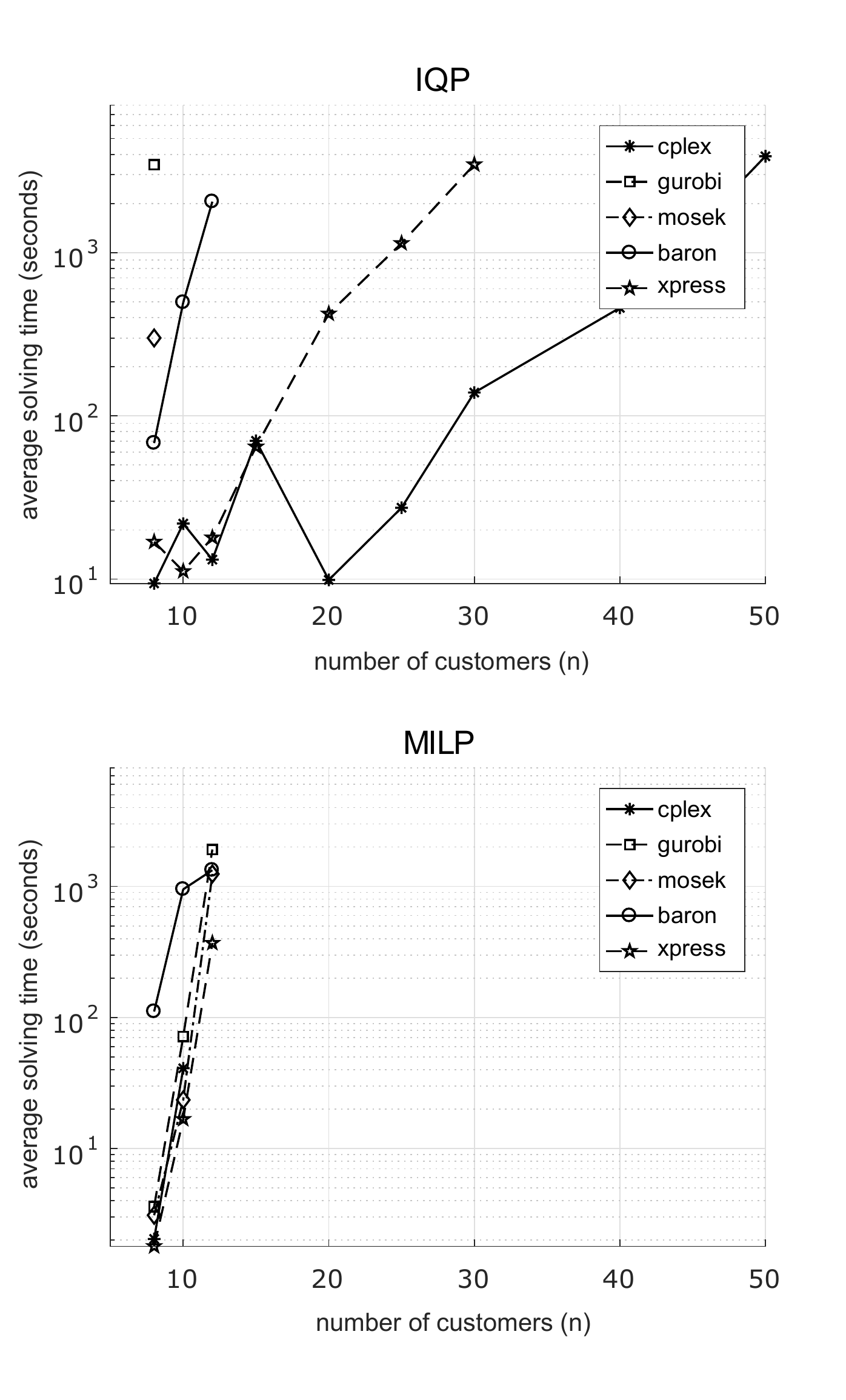}\,\,\,\,\,\,\,\,\,\,
	\includegraphics[width=2.6in]{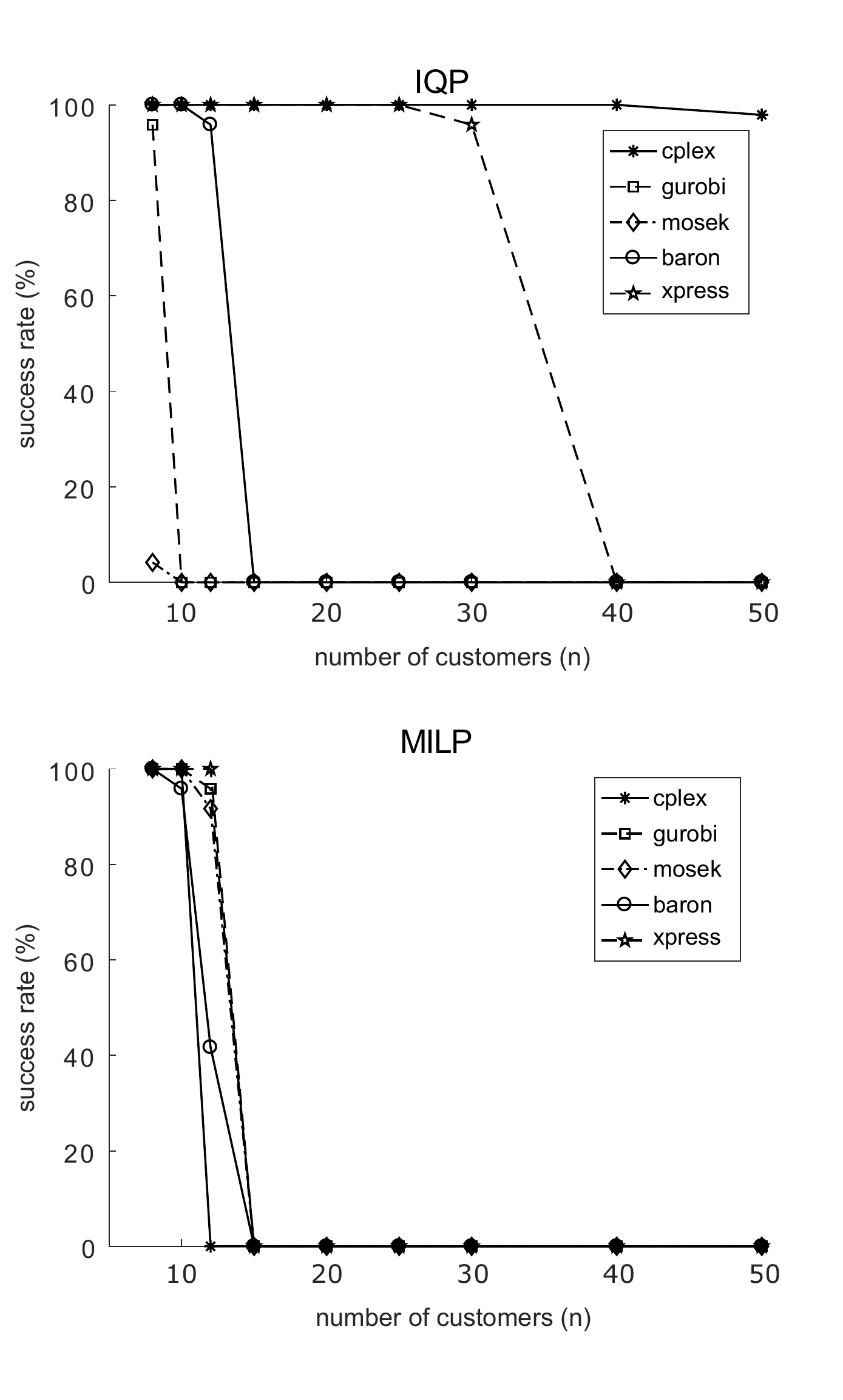}
	\caption{An illustration of the average solving time (CPU time in seconds) (left) and  success rate to reach within a prespecified gap of $3.5\%$ within the alotted time (right). Those computations that exceeded two hours were treated as two hours when generating the average solving time figures.}
	\label{fig:fig1}
\end{figure}

\begin{figure}[h]
	\centering
	\includegraphics[width=6.6in]{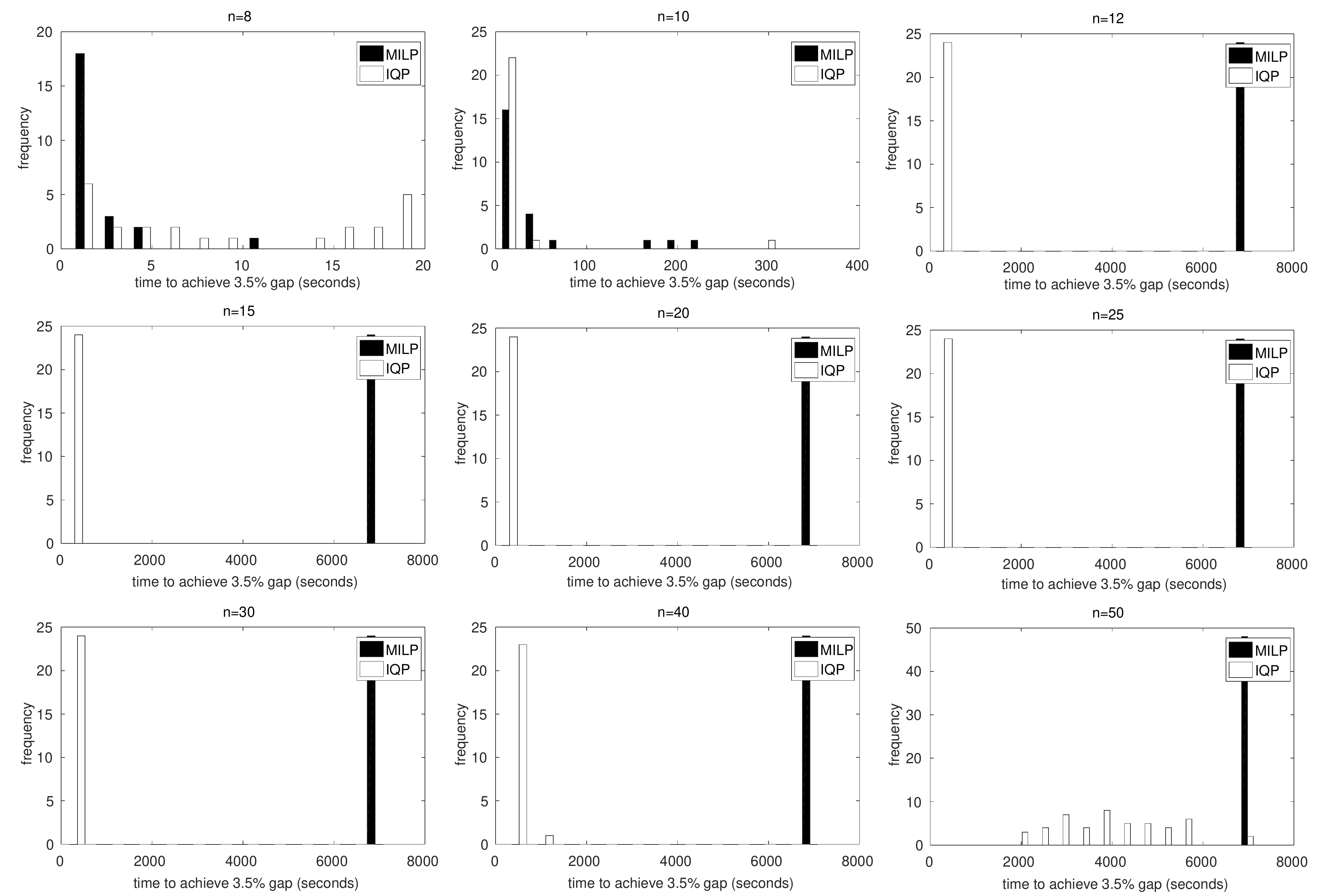}
	\caption{Comparison of solving times (in seconds) for different values of $n$ between IQP and MILP when using the CPLEX solver.}
	\label{fig:fig1CPLEX}
\end{figure}


\subsection{Evaluation While Servicing Demands Between $4-7$ Customers Using Multiple Vehicles}

Here, we demonstrate that our formulation can be used to solve problems with multiple vehicles.
Again, we solve the multi-depot version of MVPDP that was formulated using IQP using CPLEX (12.7.1) with MATLAB (R2017b) on a desktop (Intel core i5-6400, 16 gigabyte memory) for $10$ small problem instances that were randomly generated (both the positions of depots, and customers demands were uniformly sampled from a unit square, and number of customers $N$, number of vehicles $k$, and vehicle capacity $q$ were chosen between $\lbrace 4,5,6,7 \rbrace$,$\lbrace 1,2,3,4,5 \rbrace$, and $\lbrace 1,2,3,4\rbrace$, respectively).
For a given two points and the edge connecting the two, the edge cost is defined as the Euclidean distance between the two points.
The customer demands, the vehicle depots, and the computed IQP solution for the $10$ sample problem instances are depicted in Fig. \ref{fig:fig11} and Fig. \ref{fig:fig10}.
Notice that almost all of the problem instances were solved within several minutes with one exception (i.e. Fig. \ref{fig:fig62}, which took more than an hour to solve).

\begin{figure}
\centering
    \begin{subfigure}[b]{0.49\textwidth}
	\centering
	\includegraphics[width=2.25in]{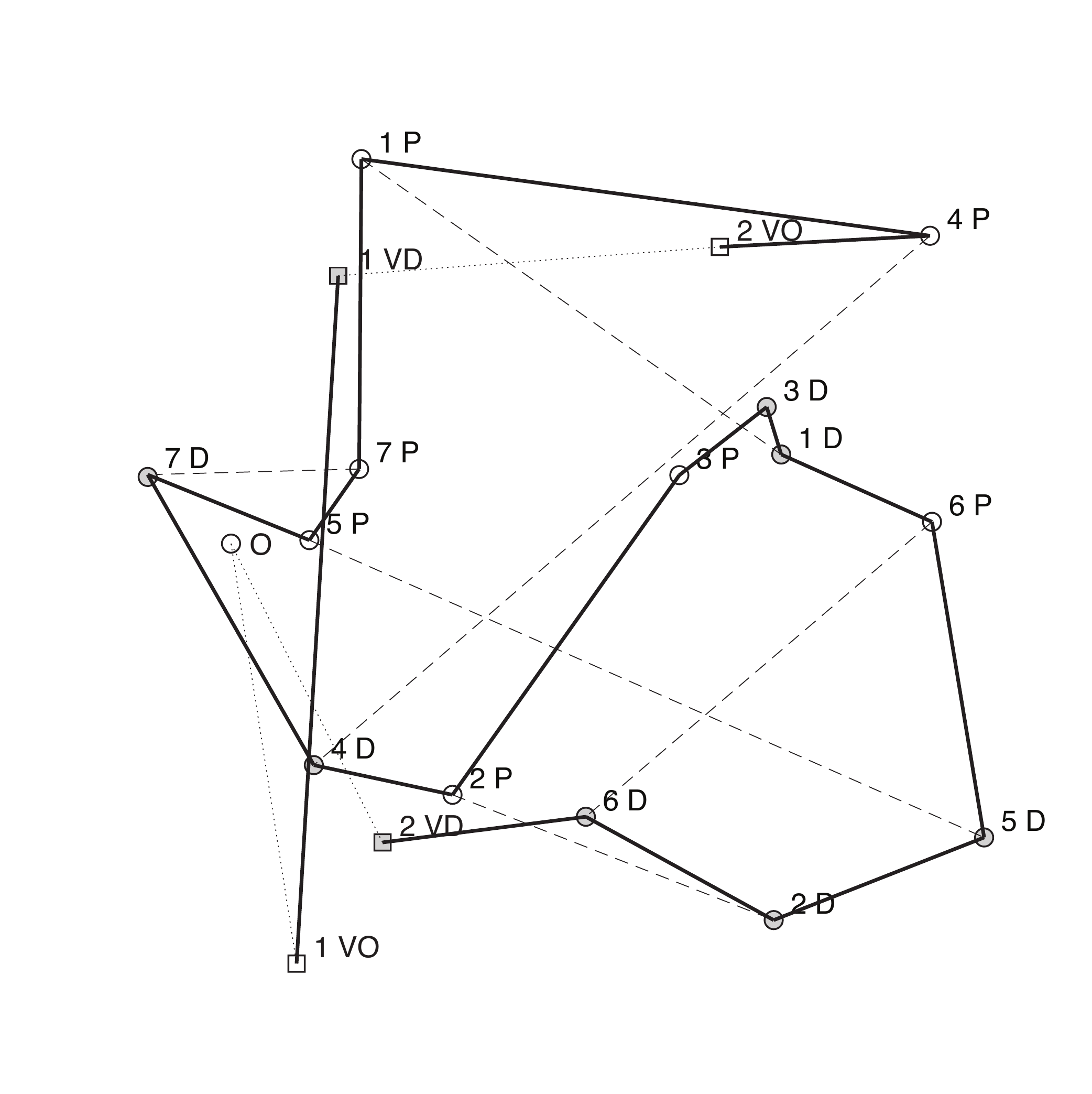}
	\caption{$n=7,\,k=2,\,q=2$, solving time: $725.85$ seconds}
	\label{fig:fig61}
    \end{subfigure}%
    \begin{subfigure}[b]{0.49\textwidth}
	\centering
	\includegraphics[width=2.25in]{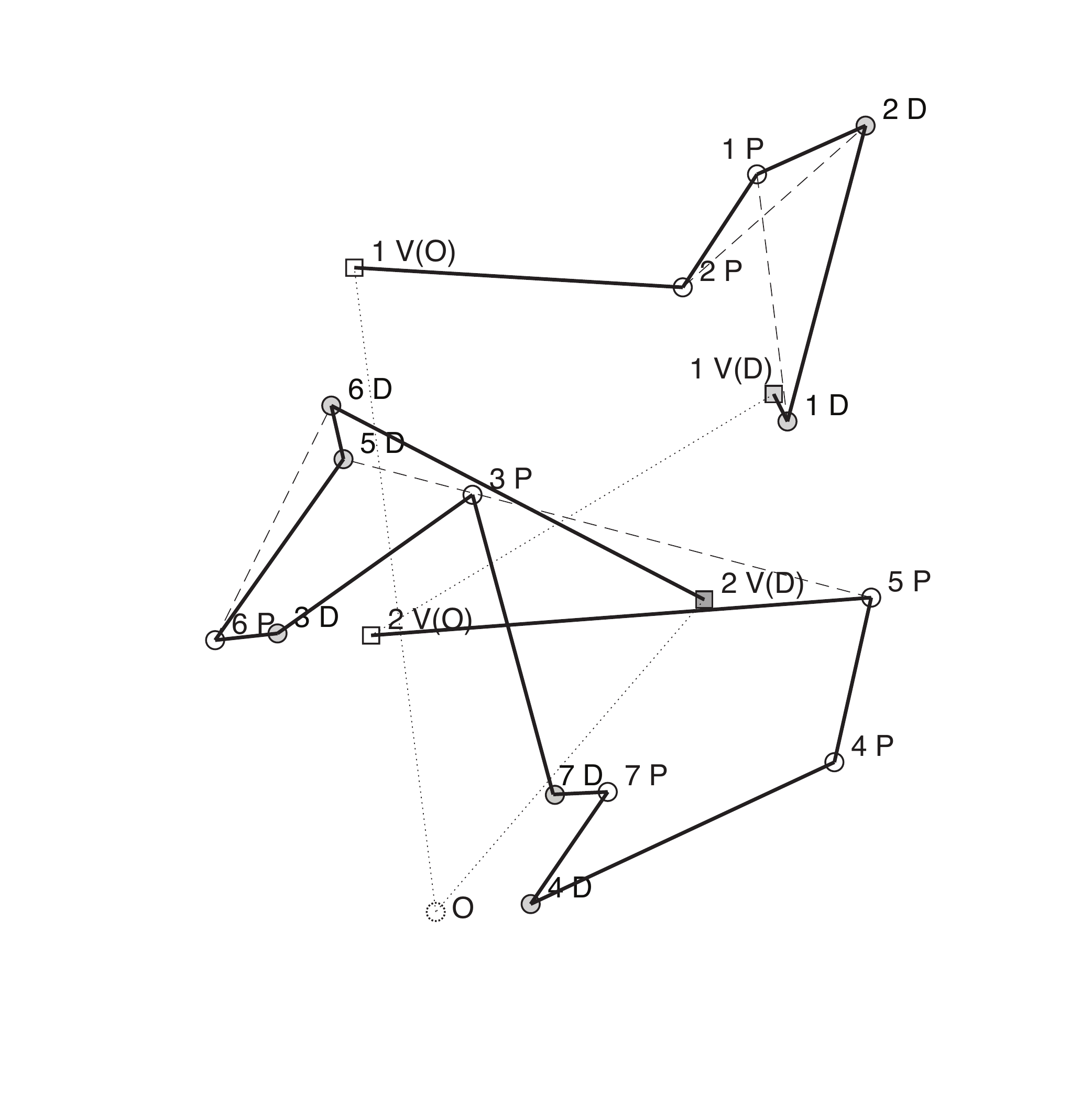}
	\caption{$n=7,\,k=3,\,q=3$, solving time: $4501.10$ seconds}
	\label{fig:fig62}
    \end{subfigure}  
    \begin{subfigure}[b]{0.49\textwidth}
	\centering
	\includegraphics[width=2.25in]{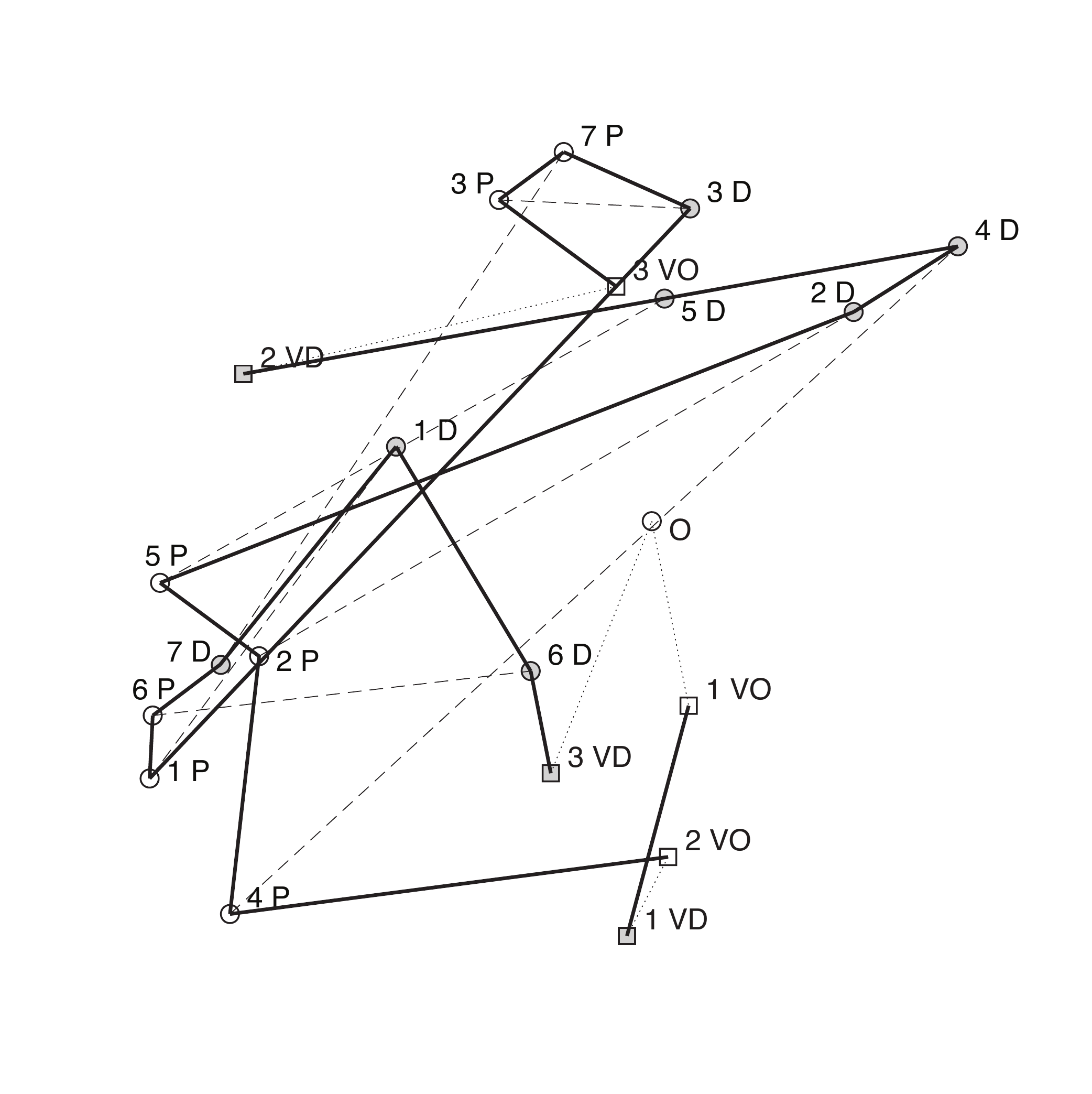}
	\caption{$n=5,\,k=2,\,q=2$, solving time: $6.93$ seconds}
	\label{fig:fig63}
    \end{subfigure} 
    \begin{subfigure}[b]{0.49\textwidth}
	\centering
	\includegraphics[width=2.25in]{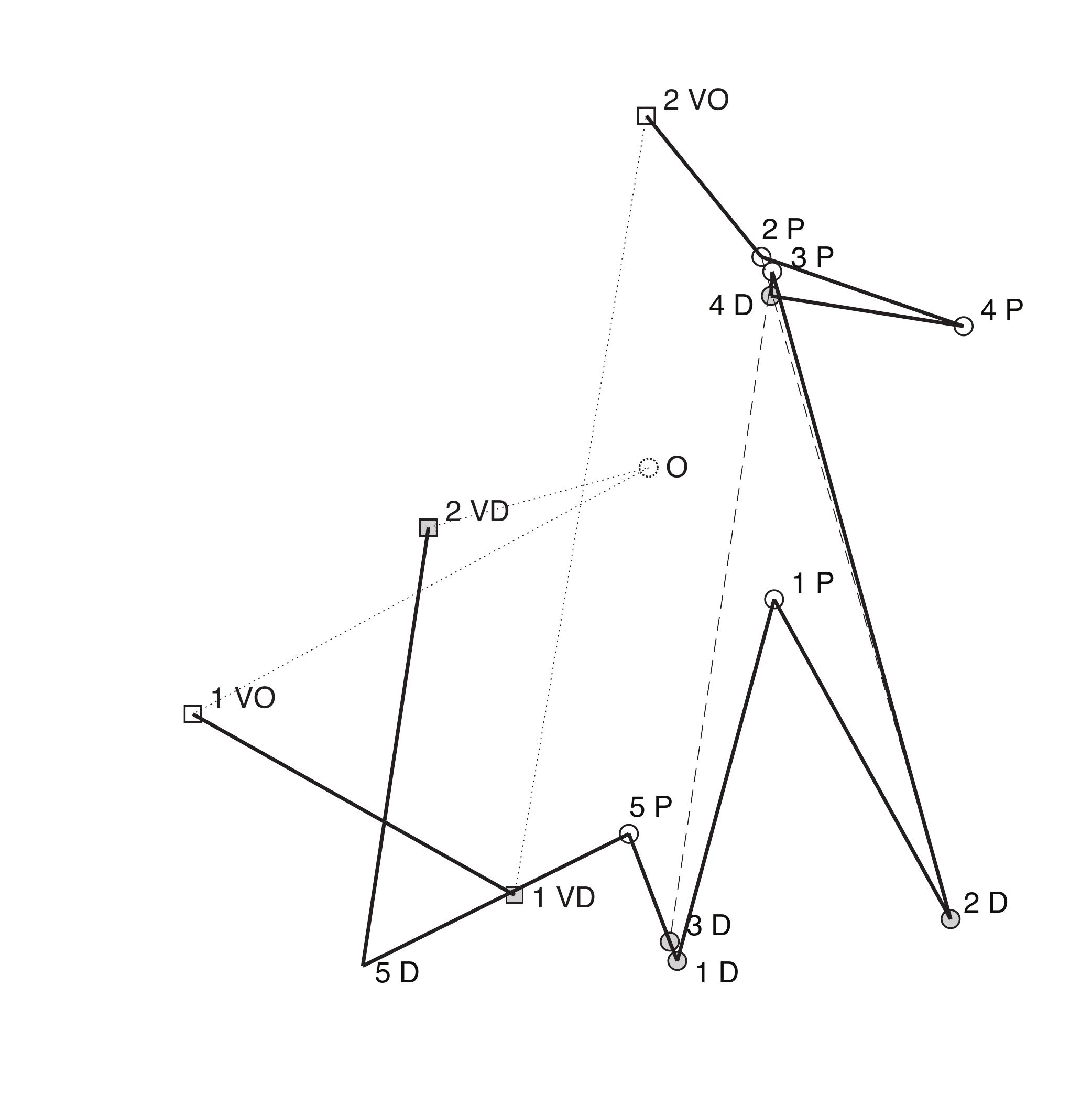}
	\caption{$n=6,\,k=4,\,q=2$, solving time: $1494.69$ seconds}
	\label{fig:fig64}
    \end{subfigure} 
    \begin{subfigure}[b]{0.49\textwidth}
	\centering
	\includegraphics[width=2.25in]{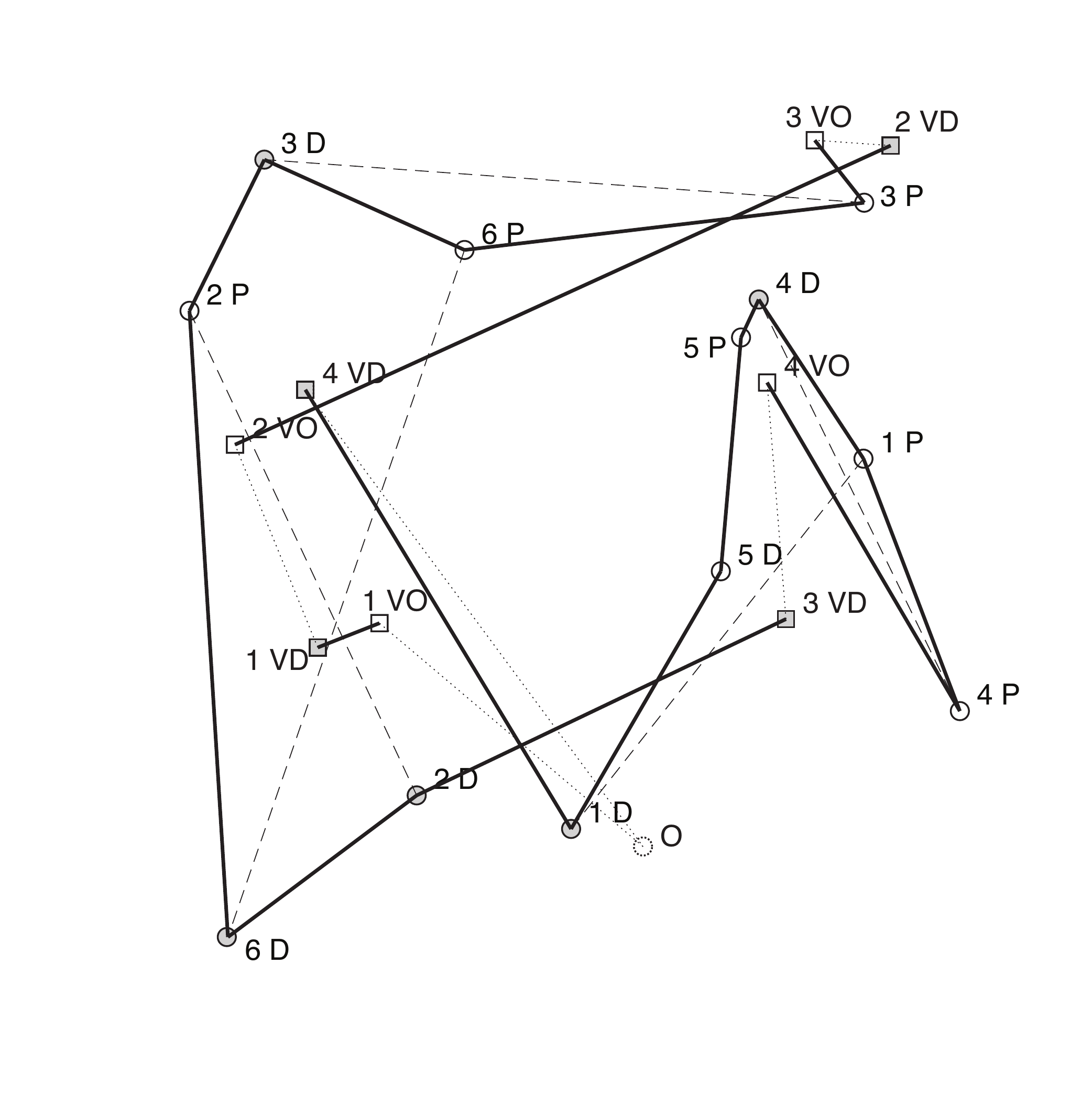}
	\caption{$n=5,\,k=2,\,q=5$, solving time: $7.65$ seconds}
	\label{fig:fig65}
    \end{subfigure}%
    \begin{subfigure}[b]{0.49\textwidth}
	\centering
	\includegraphics[width=2.25in]{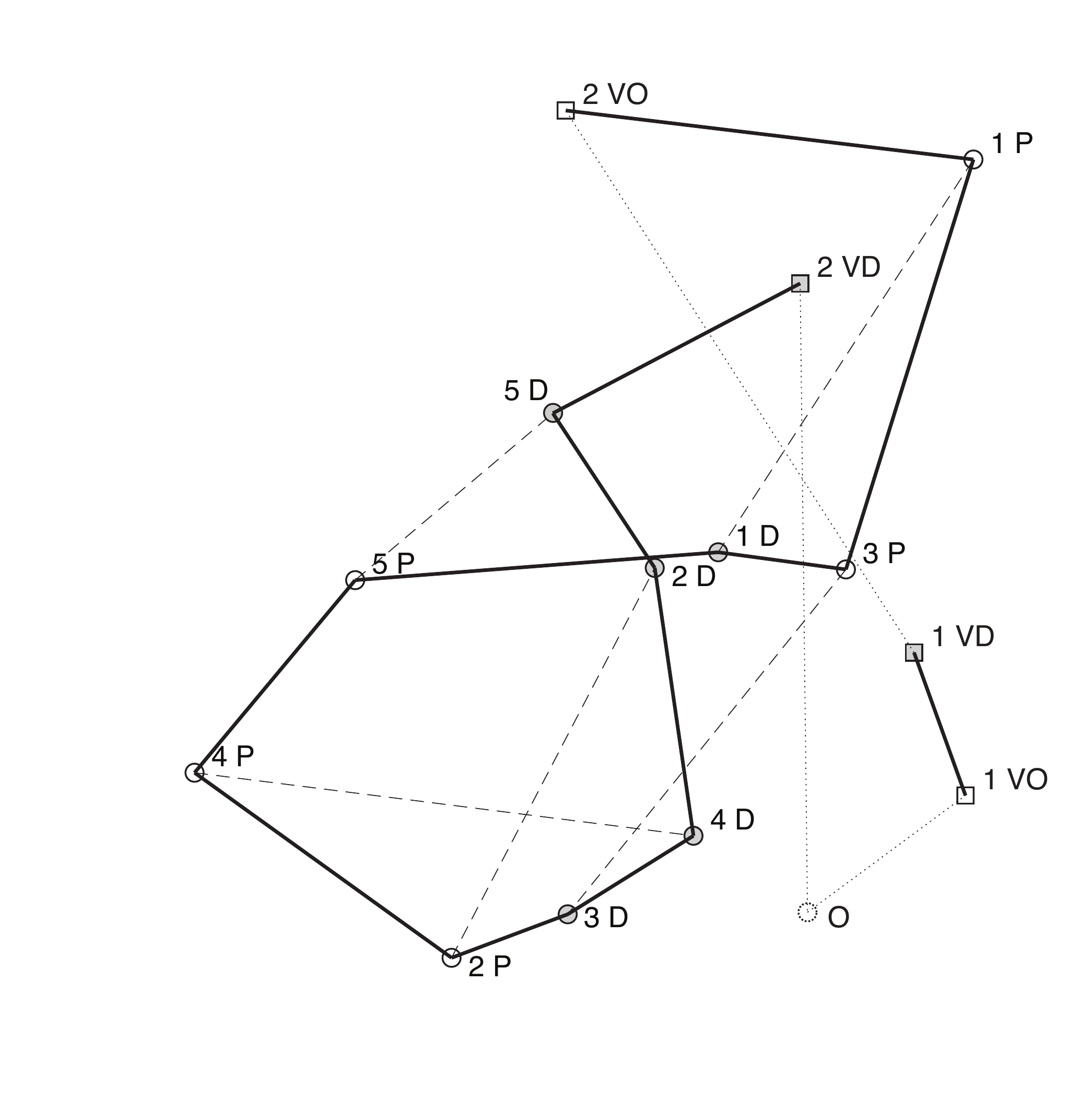}
	\caption{$n=7,\,k=1,\,q=2$, solving time: $80.02$ seconds}
	\label{fig:fig66}
    \end{subfigure} 
    \caption{An illustration of the exact solutions computed using the IQP formulation for the multi-vehicle version of the MVPDP for $6$ of the $10$ problem instances that were randomly generated. In each instance 'P' corresponds to customer's pickup demand, 'D'  corresponds to customer's delivery demand, 'VO' corresponds to vehicle's origin depot, 'VD' corresponds to vehicle destination depot,  and 'O' corresponds to the virtual node. The number $i$ preceded by the alphabet denotes either the $i$th depot or $i$th customer. The solid lines show the optimal solutions which are cycles, and dashed lines show either vehicle or customer's association, and dotted line show the virtual edges whose edge costs are zero. The empty circles are pick-up nodes, the gray circles are delivery nodes, empty squares are origin depot nodes, and the gray square are destination depot nodes.}
   \label{fig:fig11}    
\end{figure}
    
\begin{figure}
\centering    
    \begin{subfigure}[b]{0.49\textwidth}
	\centering
	\includegraphics[width=2.8in]{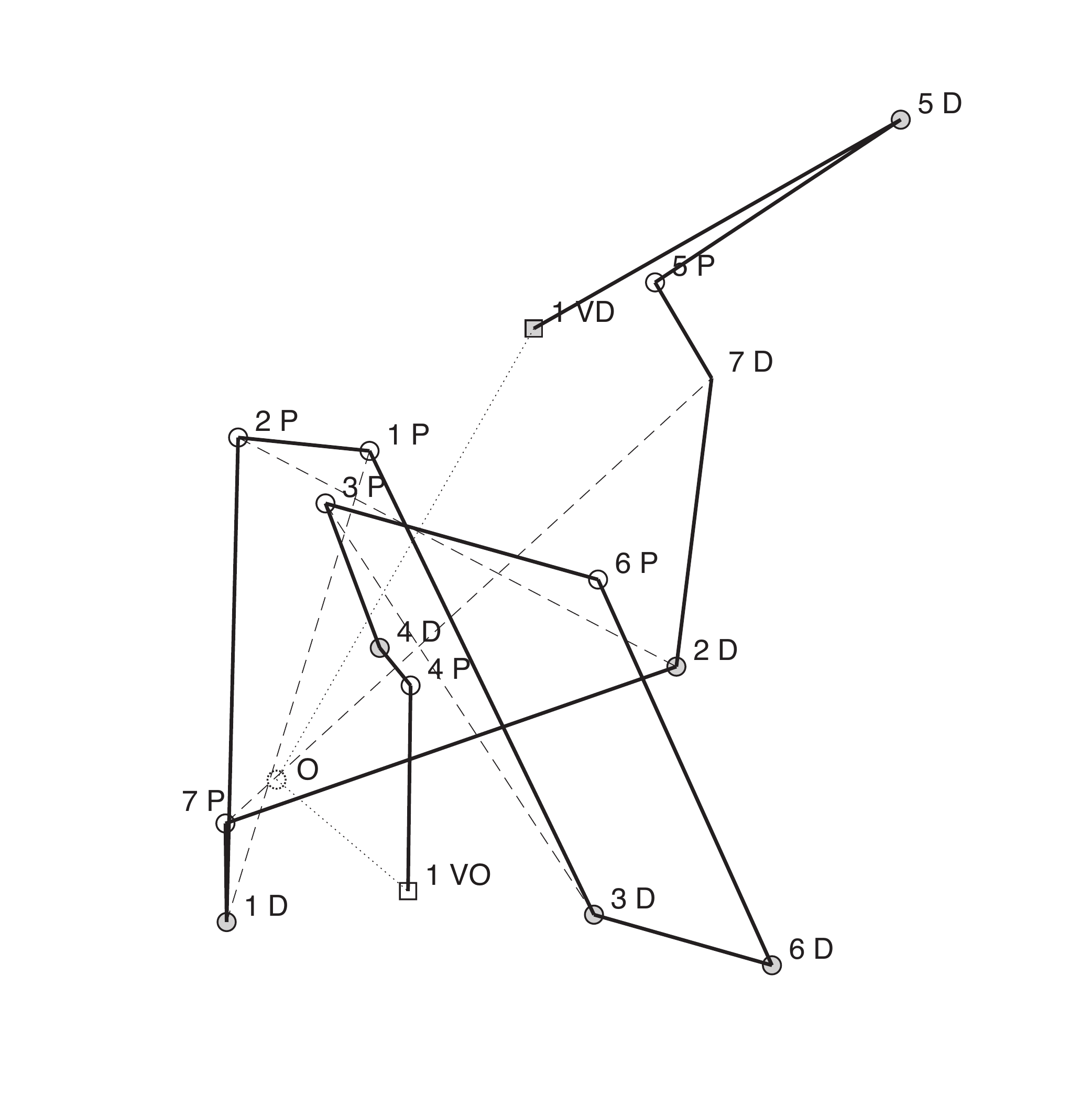}
	\caption{$n=6,\,k=3,\,q=4$, solving time: $321.92$ seconds}
	\label{fig:fig71}
    \end{subfigure}
    \begin{subfigure}[b]{0.49\textwidth}
	\centering
	\includegraphics[width=2.8in]{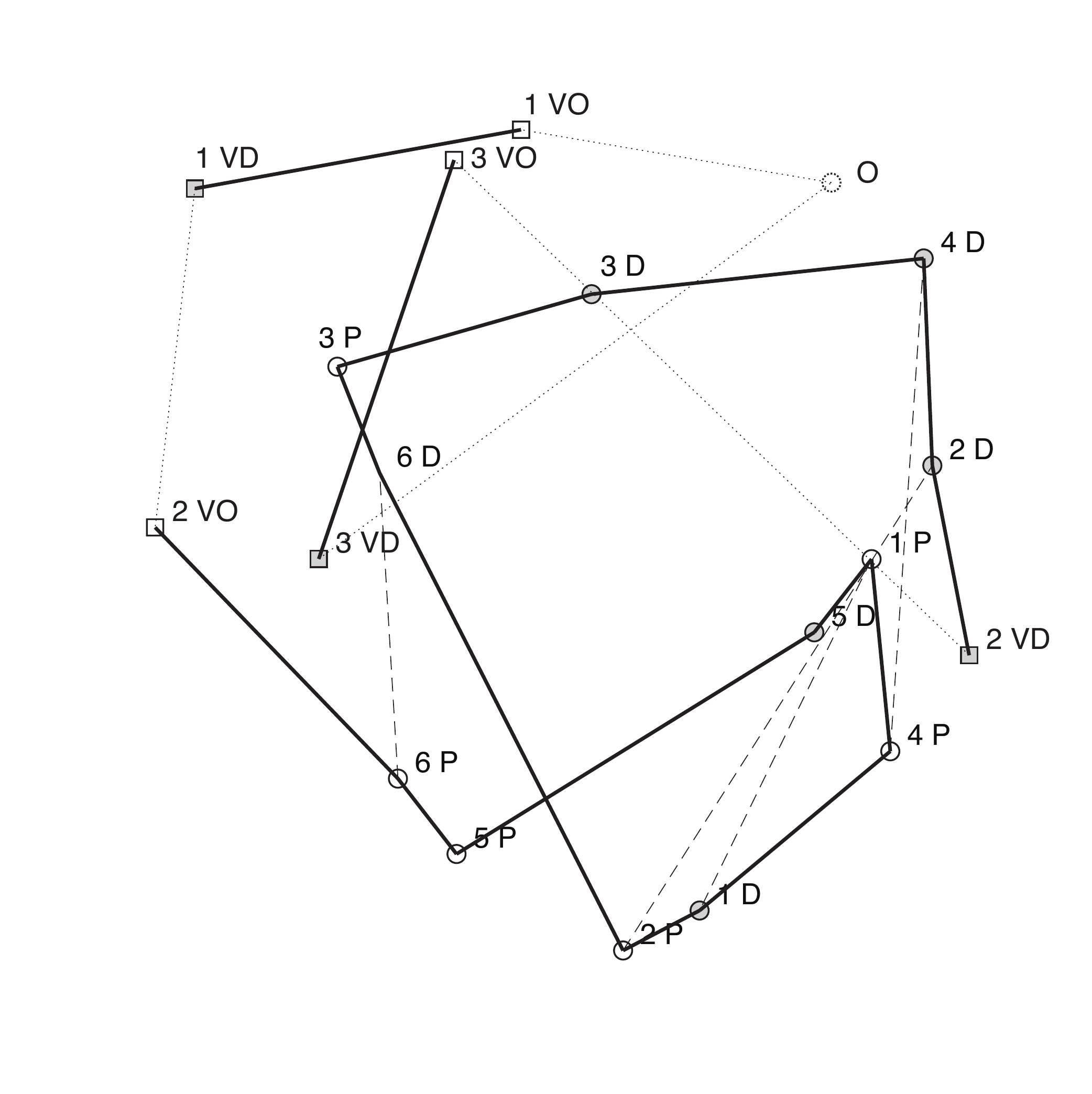}
	\caption{$n=4,\,k=2,\,q=3$, solving time: $0.26$ seconds}
	\label{fig:fig72}
    \end{subfigure} 
    \begin{subfigure}[b]{0.49\textwidth}
	\centering
	\includegraphics[width=2.8in]{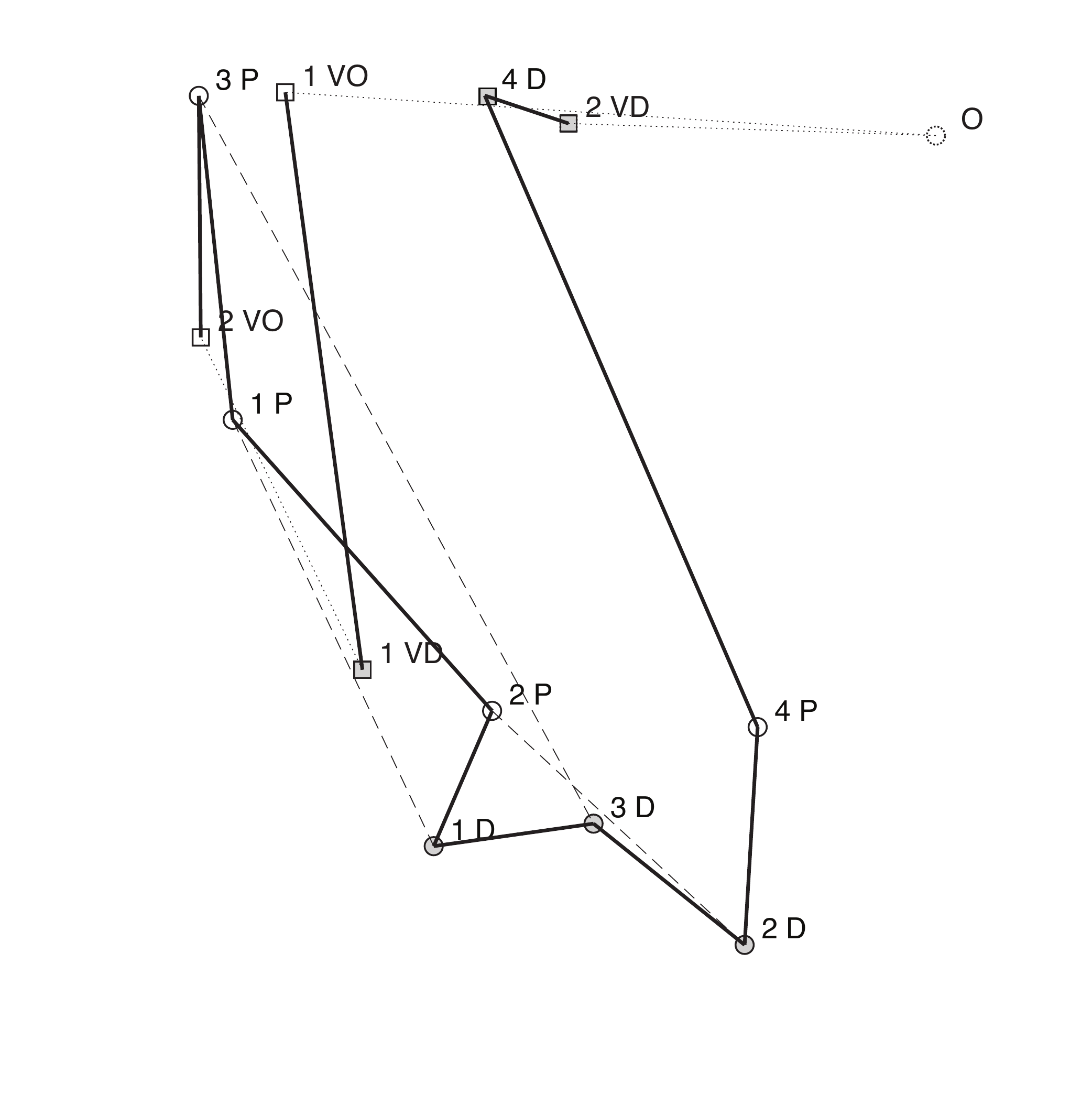}
	\caption{$n=5,\,k=5,\,q=2$, solving time: $221.71$ seconds}
	\label{fig:fig73}
    \end{subfigure}   
    \begin{subfigure}[b]{0.49\textwidth}
	\centering
	\includegraphics[width=2.8in]{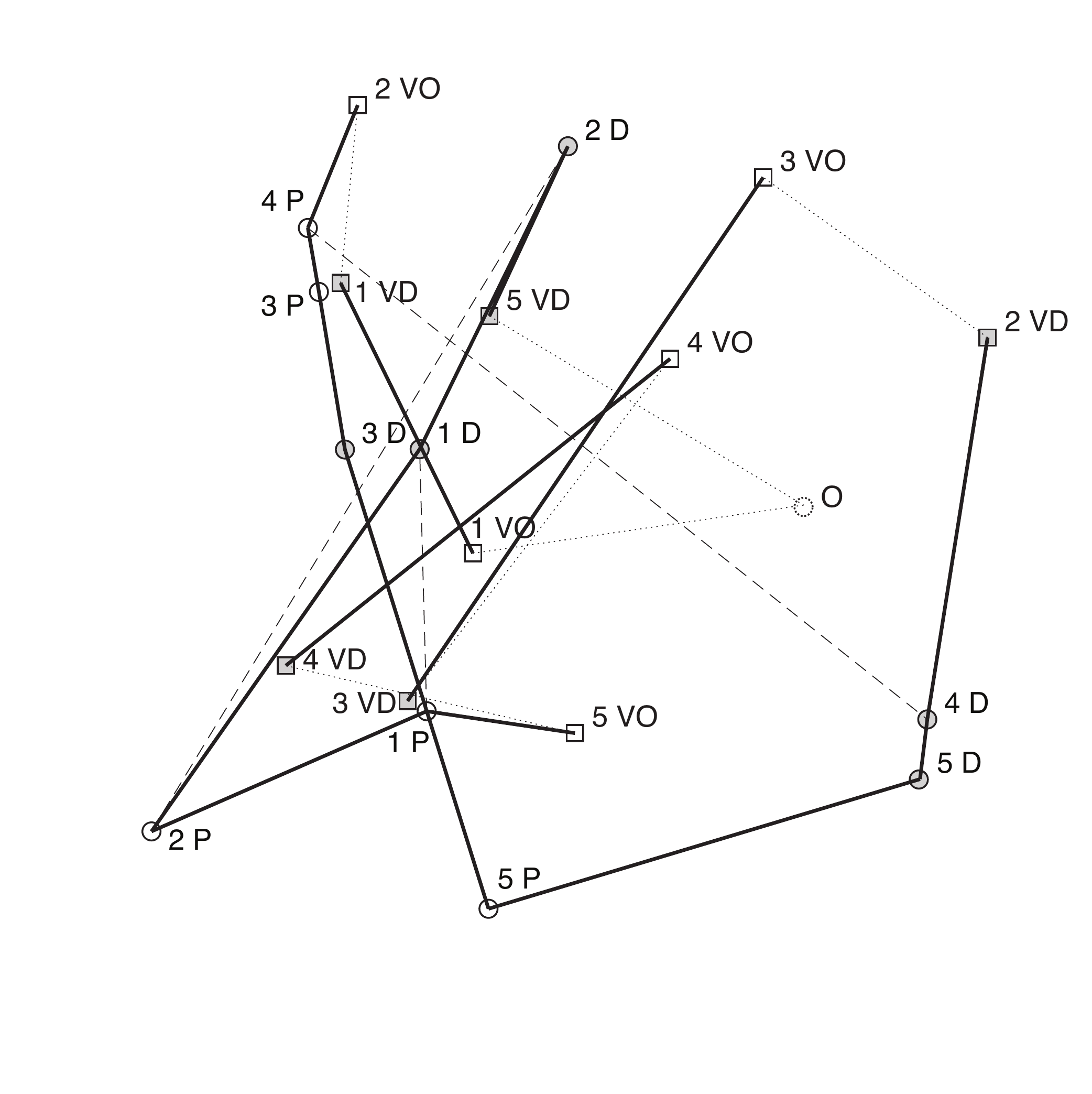}
	\caption{$n=7,\,k=2,\,q=4$, solving time: $1270.70$ seconds}
	\label{fig:fig74}
    \end{subfigure}    
    \caption{An illustration of the exact solutions computed using the IQP formulation for the multi-vehicle version of the MVPDP for $4$ of the $10$ problem instances that were randomly generated. In each instance 'P' corresponds to customer's pickup demand, 'D'  corresponds to customer's delivery demand, 'VO' corresponds to vehicle's origin depot, 'VD' corresponds to vehicle destination depot,  and 'O' corresponds to the virtual node. The number $i$ preceded by the alphabet denotes either the $i$th depot or $i$th customer. The solid lines show the optimal solutions which are cycles, and dashed lines show either vehicle or customer's association, and dotted line show the virtual edges whose edge costs are zero. The empty circles are pick-up nodes, the gray circles are delivery nodes, empty squares are origin depot nodes, and the gray square are destination depot nodes.}
 \label{fig:fig10}
\end{figure}

\subsection{Real-World Evaluation}
We further demonstrate via numerical simulation that the proposed IQP formulation can generate low-cost tours for a single yellow cab given real-world demands drawn from Manhattan in New York City.
For this simulations, we retrieved the road map of Manhattan from the Open Street Maps (OSM) database \cite{OpenStreetMap} and represent the road map as an undirected graph, $G(V,E)$. When constructing the cost matrix $C$ for the problem, each arc cost incurred between any two locations pair was obtained by computing the minimum travel distance over the road map. 
 Fig. \ref{fig:fig13}-\ref{fig:fig14} illustrates $3$ example solutions generated by the IQP formulation with $12, 15,$ and $18$ customers with a single vehicle with a maximum ride-sharing capacity of $2, 2,$ and  $3$, respectively. 
For this simulation, we have set a maximum solving time of $10$ minutes and set the MIP gap to $2\%$.
The pickup-delivery locations for customers demands for each case is chosen in sequential manner from the real yellow-cab NYC data\footnote{This TLC Trip Record Data was downloaded from ``\url{http://www.nyc.gov}'' at no cost.} during an arbitrary time-window (00:00-01:00) on 01/01/2013. 
As can be seen from Fig. \ref{fig:fig13}-\ref{fig:fig14}, the computed solution tours---drawn in a solid lines---are obtained via the CPLEX solver using our IQP formulation.
For the three problems, the corresponding state-of-the-art MILP formulation was not able to find any feasible solutions within 10 minutes. 
The simulation results in the current section reiterates our finding from the previous set of simulations that, our proposed IQP formulation can be applied to generate feasible solutions of with at most $2\%$ MIP gap within the alotted time for problems on which existing state-of-the-art solvers are unable to generate even a feasible solution.
 
\begin{figure}
\centering    
    \begin{subfigure}{\textwidth}
	\centering
	\includegraphics[width=5.5in]{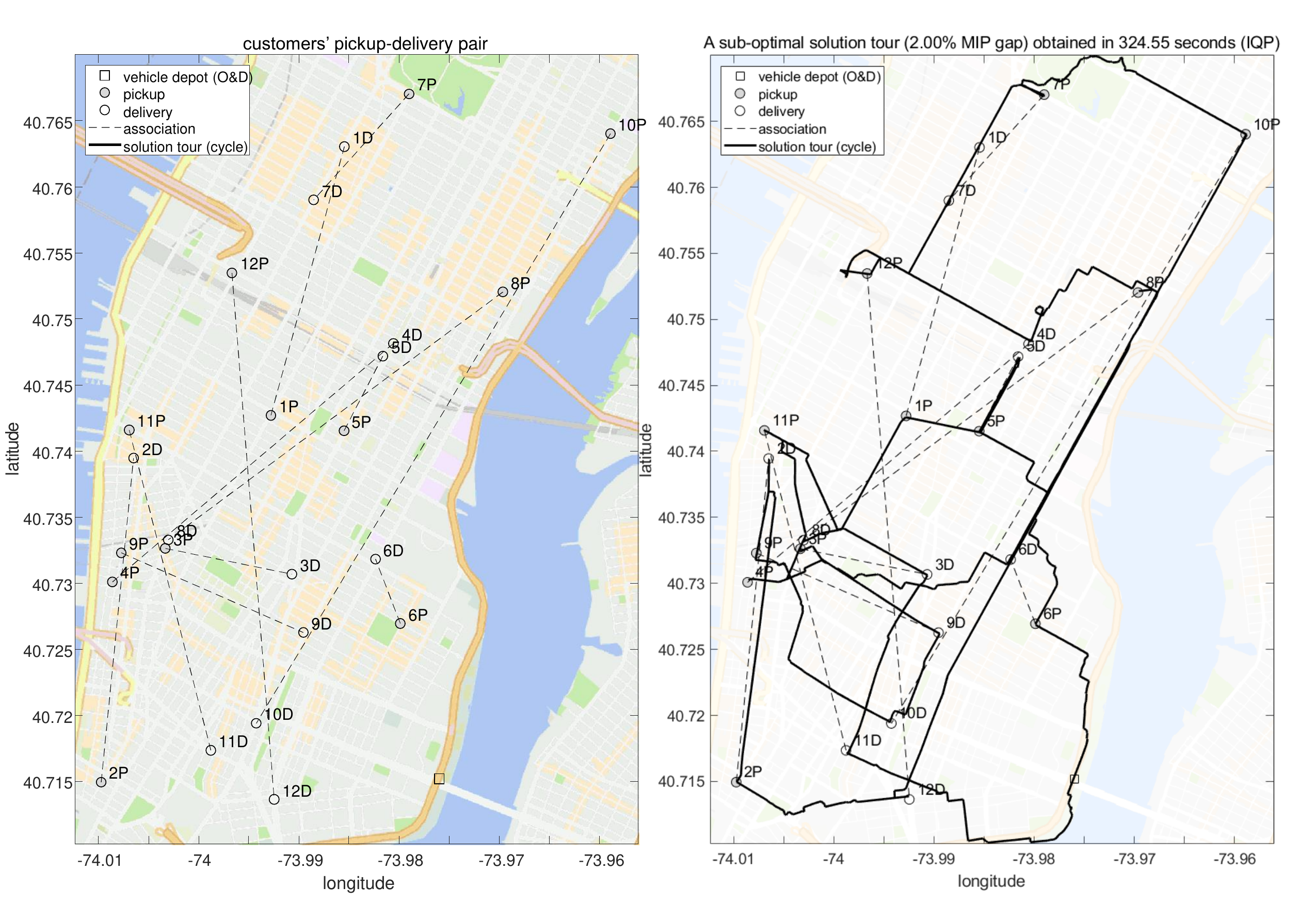}
	\caption{$n=12,\,k=1,\,q=3$}
	\label{fig:fig131}
    \end{subfigure}
    \begin{subfigure}{\textwidth}
	\centering
	\includegraphics[width=5.5in]{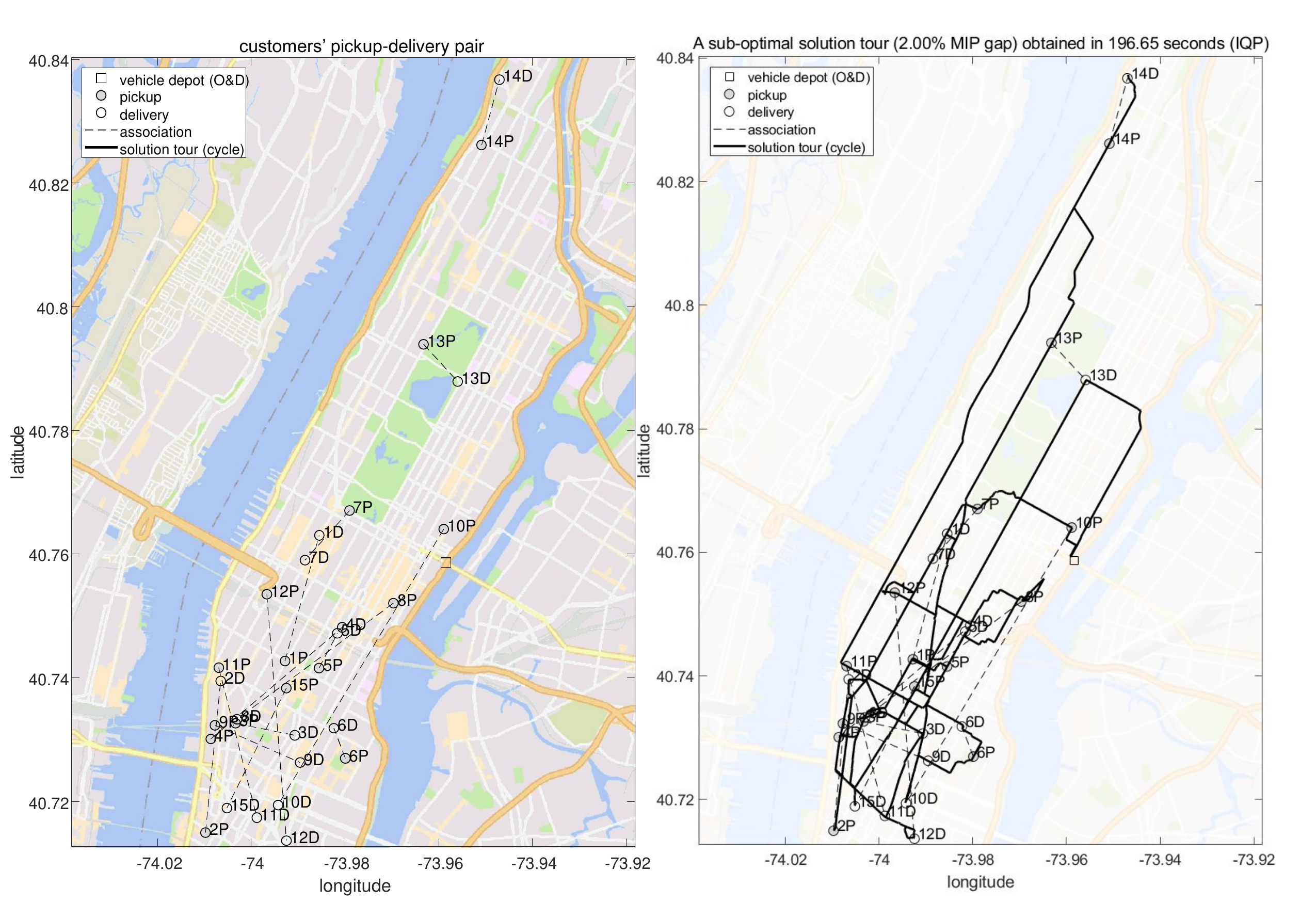}
	\caption{$n=15,\,k=1,\,q=3$}
	\label{fig:fig132}
    \end{subfigure} 
    \caption{
    An illustration of the sub-optimal solutions (right) with guaranteed $2\%$ MIP gap generated using the IQP formulation where customers' demands (left) were acquired from a real-world yellow cab data during $00:00-01:00$ on $01/01/2013$ from New York City.
    Notice that for each simulation the vehicle's origin and destination depots were chosen at random and coincide with one another.
    Solutions from the state-of-the-art MILP formulations were not displayed since no feasible solution was found within the maximum solving time of $10$ minutes.
    }
 \label{fig:fig13}
\end{figure}

\begin{figure}
    \centering
        \begin{subfigure}[b]{\textwidth}
	\centering
	\includegraphics[width=6in]{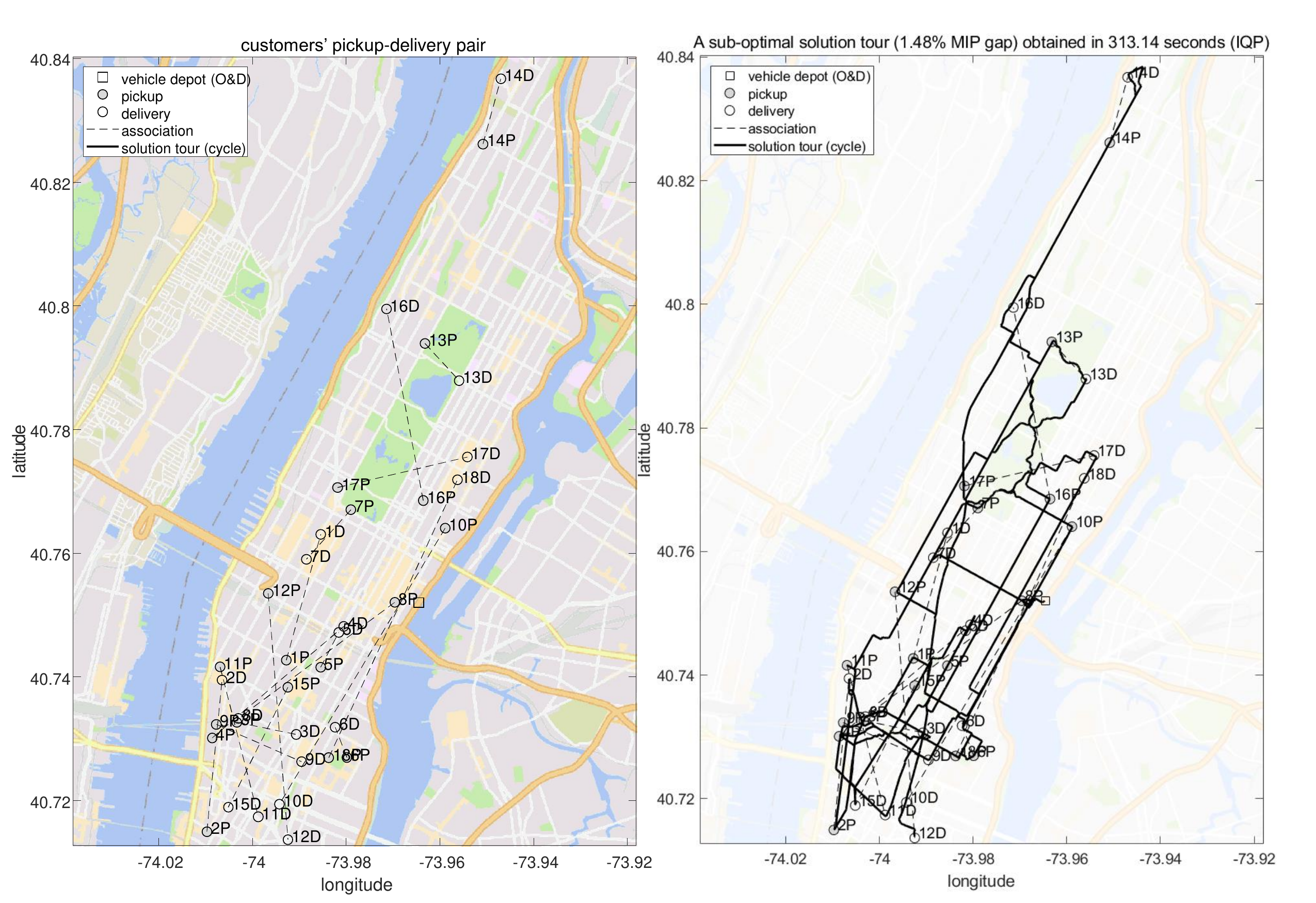}
	\caption{$n=18,\,k=2,\,q=2$}
	\label{fig:fig141}
    \end{subfigure}   
     \caption{An illustration of the sub-optimal solutions (right) with guaranteed $2\%$ MIP gap generated using the IQP formulation where customers' demands (left) were acquired from a real-world yellow cab data during $00:00-01:00$ on $01/01/2013$ from New York City.
    Notice that for each simulation the vehicle's origin and destination depots were chosen at random and coincide with one another.
    Solutions from the state-of-the-art MILP formulations were not displayed since no feasible solution was found within the maximum solving time of $10$ minutes.}
 \label{fig:fig14}   
\end{figure}

\section{Conclusions}
\label{sec6}
This paper develops an Integer Quadratic Programming based formulation to the Multiple Vehicle Pick and Delivery Problem which can be made convex (when the integer variables are relaxed).
This paper also demonstrates the superior computational efficacy of this formulation when compared to existing state of the art methods on various simulated and real-world experiments. 
The focus of this paper is on problems where constraints that may reduce the size of the search space (e.g. time window constraints) may be unavailable.
However the introduction of such constraints have improved the efficacy of existing state-of-the-art Mixed Integer Linear Programming methods.
Future work will explore techniques to introduce such time-window constraints to increase the computational efficacy of the technique proposed in this paper. 

\begin{appendices}
\numberwithin{equation}{section}
\section{Relevant Results from Algebraic Graph Theory}
\renewcommand{\theequation}{A.\arabic{equation}}

This appendix describes relevant results from algebraic graph theory.
We begin with the following definition of graph isomorphism. 

\begin{definition}[Graph Isomorphism]
	Two graphs $G^1(V^1,E^1)$ and $G^2(V^2,E^2)$ are said to be isomorphic, denoted as $G^1 \cong G^2$, if there is a bijection $\psi:V^1 \rightarrow V^2$ such that $(u,v) \in E^1$ if and only if $(\psi(u),\psi(v)) \in E^2$.
	\label{def1}
\end{definition}

In fact, all cycle graphs can be related to one another: 
\begin{proposition}
	Any two cycle graphs with at least two vertices are isomorphic to each other.
	\label{isograph}
\end{proposition}
\begin{proof}
	First, note that every vertex in a cycle graph has degree $2$, with $1$ in-degree and $1$ out-degree such that, without loss of generality, given any two cycle graphs $G_1(V_1,E_1)$, $G_2(V_2,E_2)$ with $n$ nodes, there are sets of directed edges, $\lbrace \lbrace (u_i,u_{i+1})\rbrace_{i=1}^{n-1},(u_n,u_1) \rbrace $, $\lbrace \lbrace (v_i,v_{i+1})\rbrace_{i=1}^{n-1},(v_n,v_1) \rbrace $, respectively. Then for every $u_i \in V_1$ and $v_i \in V_2$, there exist one-to-one maps $\psi_1:V_1 \rightarrow V_2$, $\psi_2:V_2 \rightarrow V_1$ respectively such that, 
	\[
	\psi_1(u_i) = \psi_2(v_i)=
	\begin{cases}
	i+1 & \textup{if } i \leq n-1 \\
	1 & \textup{if } i = n
	\end{cases}
	\]
	Thus, if $(u_i,u_j) \in E_1$ then $(v_i=\psi_2^{-1}\psi_1(u_i),v_j=\psi_2^{-1}\psi_1(u_j)) \in E_2$ and converse is also true trivially. Then the proposition is immediate by the Definition \ref{def1}.
	%
	%
	%
	%
\end{proof}
The following theorem can be used to establish Corollary \ref{col1}, which states that one can shift all the eigenvalues of any matrix to be non-negative by adding constant terms on the diagonals of the matrix.
\begin{theorem}[Gershgorin's Disk Theorem \cite{bell1965gershgorin}]
	For any square $A$, if $\lambda$ is one of the eigenvalues of $A$ then
	\[
	\left| \lambda - A_{ii} \right| \leq \min \left\{ \sum_{j=1, j\neq i}^{n} \left| A_{ij} \right|
	,\,
	\sum_{j=1,j\neq i}^{n} \left| A_{ji} \right|
	\right\}
	,\,\,\,\,\,\,\,\,\,\,\,\textup{for some } i \in \lbrace 1,2,\dots,n \rbrace.
	\]
	\label{thm2}
\end{theorem}
\begin{corollary}
	For a given matrix $A$, define a matrix $B$ of the same size as:
	\[
	B_{ij} = 
	\begin{cases}
	A_{ij} + \min\left\{ \sum_{j=1, j\neq i}^{n} \left| A_{ij} \right|
	,\,
	\sum_{j=1,j\neq i}^{n} \left| A_{ji} \right|
	\right\} + \alpha,  & \textup{if } i = j, \\
	A_{ij}, & \textup{otherwise},
	\end{cases}
	\]
	For each $\alpha \geq 0$ ($\alpha > 0$), $B \succeq 0$ ($B\succ 0)$.
	\label{col1}
\end{corollary}
The proof of the corollary is immediate from Theorem \ref{thm2} because a given matrix is positive definite (positive semi-definite) if and only of its eigenvalue are positive (non-negative). 

The proof of Proposition \ref{thm1} depends on the following proposition.
\begin{proposition}[PSD test using Schur's complement \cite{haynsworth1968schur}]
	For any symmetric matrix $M$ of the form
	\[
	M = 
	\left[
	\begin{array}{c|c}
	A & B \\
	\hline
	B^{\top} & C
	\end{array}
	\right],
	\]
	if $A$ is invertible and $A \succeq 0$ then  $M \succeq 0$ if and only if $C - B^{\top}A^{-1}B \succeq 0$.
	\label{schur}
\end{proposition}




\section{The Proof of Lemma \ref{mainlem}}
\renewcommand{\theequation}{B.\arabic{equation}}

This appendix describes the proof of Lemma 4.5, which depends on several propositions.
Let $\beta_i$ be defined recursively by:
\begin{equation}
\beta_{i+1} =\beta_0-\frac{1}{4 \beta_i},\,\,\,\,\, i =0,1,2,\dots,
\label{beta_def}
\end{equation}
where $\beta_0 = \gamma$.
The proof of Proposition \ref{lem1} depends on the following proposition.

\begin{proposition}
	Let $\beta_i$ be defined as in \eqref{beta_def}. 
	if $\beta_0\geq1$ then $\beta_{i}>\frac{\sqrt{\beta_0^2-1}+\beta_0}{2}$ for all $i \geq 0$.
	\label{prop3}
\end{proposition}
\begin{proof}[Proof of Proposition \ref{prop3}]
	We prove this proposition by induction.
	Notice that $\sqrt{\beta_0^2-1}\geq 0$, it is obvious that
	$$0 \leq \frac{\sqrt{\beta_0^2-1}+\beta_0}{2}<\beta_0 $$
	Assume $\beta_k>\frac{\sqrt{\beta_0^2-1}+\beta_0}{2}$ where $k\geq0$, and we need to show $\beta_{k+1}>\frac{\sqrt{\beta_0^2-1}+\beta_0}{2}$. By \eqref{beta_def}, we have:
	\begin{align}
	\beta_{k+1} =\beta_0-\frac{1}{4 \beta_k} & >\beta_0-\frac{1}{2\sqrt{2\beta_0^2-1}+2\beta_0}\nonumber\\
	&=\frac{2\beta_0\sqrt{2\beta_0^2-1}+2\beta_0^2-1}{2\sqrt{2\beta_0^2-1}+2\beta_0}\nonumber\\
	&= \frac{\sqrt{\beta_0^2-1}+\beta_0}{2}.\nonumber
	\end{align}
	
\end{proof}
The proof of Proposition \ref{thm1} depends on the following proposition.
\begin{proposition}
	For each $\epsilon > 0$, $\beta_0\geq 1+\epsilon$ implies
	\begin{equation}
	\beta_0-\frac{1}{4\beta_0}-\lim_{l\rightarrow \infty}\sum_{i=1}^l\frac{1}{4^{i+1}\beta_i\prod_{j=0}^{i-1}\beta_j^2}>0. \label{lem1_ineq}
	\end{equation}
	with $\beta_i$  defined as in \eqref{beta_def}.
	\label{lem1}
\end{proposition}

\begin{proof}[Proof of Proposition \ref{lem1}]
	Consider $\beta_0 > 1$.
	Based on Proposition \ref{prop3}, for each $i \geq 1$ we have:
	$$\beta_i > \frac{\sqrt{\beta_0^2-1}+\beta_0}{2} = \frac{1}{2} + \frac{\sqrt{\beta_0^2-1}+\beta_0-1}{2}. $$
	Let $\tau:= \frac{\sqrt{\beta_0^2-1}+\beta_0-1}{2}$, then $\tau >0$ and $\beta_i > \frac{1}{2}+\tau>0$. 
	For the sake of convenience, let us define a sequence $(P_n)_0^\infty$ by
	$$P_{i+1} = P_{i}\frac{1}{4\beta_i\beta_{i+1}},\,\,i=0,\,1,\dots, \text{ with } P_0 = \frac{1}{4\beta_0}$$
	Since $P_n>0$ for each $n$, it suffices to show that \eqref{lem1_ineq} holds as $l \rightarrow +\infty$.
	Notice that $4\beta_i\beta_{i+1}>(1+2\tau)^2$, then for $i>0$, $P_i<P_0\frac{1}{(1+\tau)^2} $. Thus we have:
	\begin{align}
		\lim_{l \rightarrow \infty}\sum_{i=1}^{l}P_i &< P_0\left(1+\frac{1}{(1+2\tau)^2}+\frac{1}{(1+2\tau)^4} + \dots\right)\nonumber\\
		& = \lim_{k\rightarrow+\infty}P_0\frac{1\times\left(1-\left(\frac{1}{(1+2\tau)^2} \right)^k \right)}{1-\frac{1}{(1+2\tau)^2}}\nonumber\\
		& = P_0\frac{(1+2\tau)^2}{(1+2\tau)^2-1}\nonumber.
	\end{align}
	Hence,
	$$\beta_0-\frac{1}{4\beta_0}-\lim_{l \rightarrow \infty}\sum_{i=1}^{l}\frac{1}{4^{i+1}\beta_i\prod_{j=0}^{i-1}\beta_j^2} > \beta_0-\frac{1}{4\beta_0}\frac{(1+2\tau)^2}{(1+2\tau)^2-1}.$$
    We note that $\beta_0-\frac{1}{4\beta_0}\frac{(1+2\tau)^2}{(1+2\tau)^2-1} > 0$ implies \eqref{lem1_ineq}. Also, by substituting $\tau$ with the function of $\beta_0$, it can be verified that if $\beta_0 >1$ then 
    $\beta_0-\frac{1}{4\beta_0}\frac{(1+2\tau)^2}{(1+2\tau)^2-1}>0$, which provides our sufficient condition.
\end{proof}

	We have verified though a suite of numerical simulations that $\beta_0$ can be set as low as 1, for which \eqref{lem1_ineq} holds for $l$ up to $5\times10^7$. 
	
The proof of Lemma \ref{mainlem} depends on the following proposition.
\begin{proposition}
	For a given symmetric cost matrix matrix $C \in \mathbb{R}^{v \times v}$ that is positive semi-definite, if $\gamma > 1$  then $Q\succeq 0$.
	\label{thm1}
\end{proposition}
\begin{proof}[Proof of Proposition \ref{thm1}]
	\
	Consider a real square matrix $\Gamma_i$:
	\[
	\Gamma_i = \left[\begin{array}{c|c}
	A_i & b_i \\
	\hline
	b_i^{\top} & D_i 
	\end{array}
	\right]
	\]
	We note that by 
	Proposition \ref{schur} (PSD test using Schur’s complement), $\Gamma_i \succeq 0$ if and only if $A_i \succeq 0$ and $D_i - b_i^{\top} A_i^{-1} b_i \succeq 0$. Let us define a sequence of matrices recursively by
	\[
	\Gamma_{i+1}:= D_i - b_i^{\top} A_i^{-1} b_i,
	\]
	where $\Gamma_{0} := Q$. We note that $i \in \lbrace 0,1,\dots, \frac{n}{k}-1 \rbrace$. Then it is a routine to verify that for each $i$
	\[
	\Gamma_i 
	=  \left[\begin{array}{c|ccccc}
	\beta_i \overline{C} & C/2 & 0 & \cdots & 0 & (-1)^i\frac{C}{2^{i+1}\prod_{j=1}^{i-1}\beta_j}  \\ \hline
	C^{\top}/2 & \beta_i \overline{C} & C/2 & 0 & \cdots & 0 \\
	0 & C^{\top}/2 & \ddots &\ddots & \ddots & \vdots \\
	\vdots & \ddots & \ddots & \ddots & C/2 & 0 \\  
	0 & \cdots & 0 & C^{\top}/2 & \beta_i \overline{C} & C/2 \\
	(-1)^i\frac{C^{\top}}{2^{i+1}\prod_{j=1}^{i-1}\beta_j} & 0 & \cdots & 0 & C^{\top}/2 & C \left(\beta_0 - \frac{1}{4\beta_0} -\sum_{l=1}^{i-1}\frac{1}{4^{l+1}\beta_i\prod_{j=0}^{l-1}\beta_j^2} \right)
	\end{array}\right]
	\]
	where $A_i =  \beta_i \overline{C}$, $b_i = \left[C/2 ,\, 0 ,\, \cdots  0 ,\, (-1)^i\frac{C}{2^{i+1}\prod_{j=1}^{i-1}\beta_j} \right]$, and $D_i$ is the block matrix from the bottom left.
	Noting that $C$ is symmetric, and after applying the Schur's complement recursively, it can be verified that $Q \succeq 0 $ if and only if $A_i \succeq 0$ for all $i \in \lbrace 0,1,\dots,\frac{n}{k}-1 \rbrace$ and $\Gamma_{\frac{n}{k}-1} \succeq 0$. Since $\beta_i >0$ implies $A_i \succeq 0$, we only need to show that for a given condition $\Gamma_{\frac{n}{k}-1}$ is positive semi-definite. Since 
	\[
	\Gamma_{\frac{n}{k}-1} = 
	C \left( \beta_0 -\frac{1}{4\beta_0}-\sum_{i=1}^{\frac{n}{k}}\frac{1}{4^{i+1}\beta_i\prod_{j=0}^{i-1}\beta_j^2} 
	\right),
	\]
	for each $C \succeq 0$, and  $\gamma=\beta_0 > 1$, by the Proposition \ref{lem1},  $\Gamma_{\frac{n}{k}-1} \succeq 0$, which completes the proof.
\end{proof}
\begin{proof}[Proof of Lemma \ref{mainlem}]
For the first part of the proof, we consider $\gamma =0$. By plugging in $\gamma = 0$ in \eqref{qmat}, we obtain:
\begin{equation}
\left.Q\middle|_{\gamma =0}\right.= 
\underbrace{\begin{bmatrix}
	0 & C/2      & 0  & \cdots & 0 & C/2 \\
	C^{\top}/2   & 0 &  C/2 & 0 & \cdots & 0\\
	0 & C^{\top}/2 & \ddots & \ddots & \ddots & \vdots \\
	\vdots& \ddots & \ddots & \ddots & C/2 & 0 \\
	0   & \cdots      & 0  &  C^{\top}/2   &0 & C/2  \\
	C^{\top}/2 & 0      & \cdots & 0 & C^{\top}/2 & 0
	\end{bmatrix}}_{v \textup{ blocks}}
\end{equation}
which we denote by $Q_0$. We note again that
\begin{align*}
    \textup{trace}(C^{\top}X^{\top}A^{0}X) &=
    \bm{x}^{\top}Q_0\bm{x}+\bm{x}^{\top}D\bm{x} -\bm{x}^{\top}D\bm{x} =\bm{x}^{\top}(Q_0+D)\bm{x} - \bm{d}^{\top}\bm{x}
\end{align*}
where $\bm{d}$ is a vector such that, $D = \textup{diag}(\bm{d})$.
We apply the Corollary \ref{col1} multiple times to $Q_0+D$ directly to find condition for $D$ which makes all the eigenvalues of the matrix sum non-negative.
For each $i = 1,\dots,v$, and $i = v^2-v+1,v^2-v+2,\dots,v^2$, we can apply Corollary \ref{col1} to obtain the condition to ensure that non-negativity of the eigenvalue. For each row: $
d_i \geq \sum_{j = 1}^v c_{ij}
$, 
and for each column:
$
d_i \geq \sum_{j = 1}^v c_{ji}.
$
By combining the two conditions, we have
\[
d_i \geq \min \left\{ \sum_{j = 1}^v c_{ij},\,\sum_{j = 1}^v c_{ji}. \right\},\,
i = 1,\dots,v,v^2-v+1,v^2-v+2,\dots,v^2.
\]
Now for each $i = v+1,v+2,\dots,v^2-v$, it must be that 
\[
d_i \geq \frac{\sum_{j = 1}^v c_{ij}+\sum_{j = 1}^v c_{ji}}{2}. 
\]
This completes the proof for the first part of the lemma.
For the second part of the proof we consider $\gamma > 1$ and $D =0$. Then, by the given conditions, the 2nd argument is immediate by the Proposition \ref{thm1}.
\end{proof}
\end{appendices}

\bibliographystyle{ieeetr}
\bibliography{reference}

\begin{thebibliography}{10}

\bibitem{cici2014assessing}
B.~Cici, A.~Markopoulou, E.~Frias-Martinez, and N.~Laoutaris, ``Assessing the
  potential of ride-sharing using mobile and social data: a tale of four
  cities,'' in {\em Proceedings of the 2014 ACM International Joint Conference
  on Pervasive and Ubiquitous Computing}, pp.~201--211, ACM, 2014.

\bibitem{gloss2016designing}
M.~Gl{\"o}ss, M.~McGregor, and B.~Brown, ``Designing for labour: uber and the
  on-demand mobile workforce,'' in {\em Proceedings of the 2016 CHI Conference
  on Human Factors in Computing Systems}, pp.~1632--1643, ACM, 2016.

\bibitem{hernandez2009multi}
H.~Hern{\'a}ndez-P{\'e}rez and J.-J. Salazar-Gonz{\'a}lez, ``The
  multi-commodity one-to-one pickup-and-delivery traveling salesman problem,''
  {\em European Journal of Operational Research}, vol.~196, no.~3,
  pp.~987--995, 2009.

\bibitem{cordeau2003dial}
J.-F. Cordeau and G.~Laporte, ``The dial-a-ride problem (darp): Variants,
  modeling issues and algorithms,'' {\em 4OR: A Quarterly Journal of Operations
  Research}, vol.~1, no.~2, pp.~89--101, 2003.

\bibitem{cordeau2007dial}
J.-F. Cordeau and G.~Laporte, ``The dial-a-ride problem: models and
  algorithms,'' {\em Annals of operations research}, vol.~153, no.~1, p.~29,
  2007.

\bibitem{cherkesly2015population}
M.~Cherkesly, G.~Desaulniers, and G.~Laporte, ``A population-based
  metaheuristic for the pickup and delivery problem with time windows and lifo
  loading,'' {\em Computers \& Operations Research}, vol.~62, pp.~23--35, 2015.

\bibitem{hernandez2016hybrid}
H.~Hern{\'a}ndez-P{\'e}rez, I.~Rodr{\'\i}guez-Mart{\'\i}n, and J.-J.
  Salazar-Gonz{\'a}lez, ``A hybrid heuristic approach for the multi-commodity
  pickup-and-delivery traveling salesman problem,'' {\em European Journal of
  Operational Research}, vol.~251, no.~1, pp.~44--52, 2016.

\bibitem{ting2013selective}
C.-K. Ting and X.-L. Liao, ``The selective pickup and delivery problem:
  formulation and a memetic algorithm,'' {\em International Journal of
  Production Economics}, vol.~141, no.~1, pp.~199--211, 2013.

\bibitem{goksal2013hybrid}
F.~P. Goksal, I.~Karaoglan, and F.~Altiparmak, ``A hybrid discrete particle
  swarm optimization for vehicle routing problem with simultaneous pickup and
  delivery,'' {\em Computers \& Industrial Engineering}, vol.~65, no.~1,
  pp.~39--53, 2013.

\bibitem{wang2015parallel}
C.~Wang, D.~Mu, F.~Zhao, and J.~W. Sutherland, ``A parallel simulated annealing
  method for the vehicle routing problem with simultaneous pickup--delivery and
  time windows,'' {\em Computers \& Industrial Engineering}, vol.~83,
  pp.~111--122, 2015.

\bibitem{gansterer2017multi}
M.~Gansterer, M.~K{\"u}{\c{c}}{\"u}ktepe, and R.~F. Hartl, ``The multi-vehicle
  profitable pickup and delivery problem,'' {\em OR Spectrum}, vol.~39, no.~1,
  pp.~303--319, 2017.

\bibitem{ritzinger2016dynamic}
U.~Ritzinger, J.~Puchinger, and R.~F. Hartl, ``Dynamic programming based
  metaheuristics for the dial-a-ride problem,'' {\em Annals of Operations
  Research}, vol.~236, no.~2, pp.~341--358, 2016.

\bibitem{cordeau2006branch}
J.-F. Cordeau, ``A branch-and-cut algorithm for the dial-a-ride problem,'' {\em
  Operations Research}, vol.~54, no.~3, pp.~573--586, 2006.

\bibitem{psaraftis1980dynamic}
H.~N. Psaraftis, ``A dynamic programming solution to the single vehicle
  many-to-many immediate request dial-a-ride problem,'' {\em Transportation
  Science}, vol.~14, no.~2, pp.~130--154, 1980.

\bibitem{yin2016energy}
J.~Yin, T.~Tang, L.~Yang, Z.~Gao, and B.~Ran, ``Energy-efficient metro train
  rescheduling with uncertain time-variant passenger demands: An approximate
  dynamic programming approach,'' {\em Transportation Research Part B:
  Methodological}, vol.~91, pp.~178--210, 2016.

\bibitem{psaraftis1983exact}
H.~N. Psaraftis, ``An exact algorithm for the single vehicle many-to-many
  dial-a-ride problem with time windows,'' {\em Transportation science},
  vol.~17, no.~3, pp.~351--357, 1983.

\bibitem{desrosiers1986dynamic}
J.~Desrosiers, Y.~Dumas, and F.~Soumis, ``A dynamic programming solution of the
  large-scale single-vehicle dial-a-ride problem with time windows,'' {\em
  American Journal of Mathematical and Management Sciences}, vol.~6, no.~3-4,
  pp.~301--325, 1986.

\bibitem{mahmoudi2016finding}
M.~Mahmoudi and X.~Zhou, ``Finding optimal solutions for vehicle routing
  problem with pickup and delivery services with time windows: A dynamic
  programming approach based on state--space--time network representations,''
  {\em Transportation Research Part B: Methodological}, vol.~89, pp.~19--42,
  2016.

\bibitem{colorni2001modeling}
A.~Colorni and G.~Righini, ``Modeling and optimizing dynamic dial-a-ride
  problems,'' {\em International transactions in operational research}, vol.~8,
  no.~2, pp.~155--166, 2001.

\bibitem{barbato2016polyhedral}
M.~Barbato, R.~Grappe, M.~Lacroix, and R.~W. Calvo, ``Polyhedral results and a
  branch-and-cut algorithm for the double traveling salesman problem with
  multiple stacks,'' {\em Discrete Optimization}, vol.~21, pp.~25--41, 2016.

\bibitem{kalantari1985algorithm}
B.~Kalantari, A.~V. Hill, and S.~R. Arora, ``An algorithm for the traveling
  salesman problem with pickup and delivery customers,'' {\em European Journal
  of Operational Research}, vol.~22, no.~3, pp.~377--386, 1985.

\bibitem{cote2012branch}
J.-F. C{\^o}t{\'e}, C.~Archetti, M.~G. Speranza, M.~Gendreau, and J.-Y. Potvin,
  ``A branch-and-cut algorithm for the pickup and delivery traveling salesman
  problem with multiple stacks,'' {\em Networks}, vol.~60, no.~4, pp.~212--226,
  2012.

\bibitem{ruland1997pickup}
K.~Ruland and E.~Rodin, ``The pickup and delivery problem: Faces and
  branch-and-cut algorithm,'' {\em Computers \& mathematics with applications},
  vol.~33, no.~12, pp.~1--13, 1997.

\bibitem{ropke2009branch}
S.~Ropke and J.-F. Cordeau, ``Branch and cut and price for the pickup and
  delivery problem with time windows,'' {\em Transportation Science}, vol.~43,
  no.~3, pp.~267--286, 2009.

\bibitem{ropke2007models}
S.~Ropke, J.-F. Cordeau, and G.~Laporte, ``Models and branch-and-cut algorithms
  for pickup and delivery problems with time windows,'' {\em Networks},
  vol.~49, no.~4, pp.~258--272, 2007.

\bibitem{cherkesly2016branch}
M.~Cherkesly, G.~Desaulniers, S.~Irnich, and G.~Laporte, ``Branch-price-and-cut
  algorithms for the pickup and delivery problem with time windows and multiple
  stacks,'' {\em European Journal of Operational Research}, vol.~250, no.~3,
  pp.~782--793, 2016.

\bibitem{dumas1991pickup}
Y.~Dumas, J.~Desrosiers, and F.~Soumis, ``The pickup and delivery problem with
  time windows,'' {\em European journal of operational research}, vol.~54,
  no.~1, pp.~7--22, 1991.

\bibitem{bertsimas1991stochastic}
D.~J. Bertsimas and G.~Van~Ryzin, ``A stochastic and dynamic vehicle routing
  problem in the euclidean plane,'' {\em Operations Research}, vol.~39, no.~4,
  pp.~601--615, 1991.

\bibitem{furuhata2013ridesharing}
M.~Furuhata, M.~Dessouky, F.~Ord{\'o}{\~n}ez, M.-E. Brunet, X.~Wang, and
  S.~Koenig, ``Ridesharing: The state-of-the-art and future directions,'' {\em
  Transportation Research Part B: Methodological}, vol.~57, pp.~28--46, 2013.

\bibitem{berbeglia2007static}
G.~Berbeglia, J.-F. Cordeau, I.~Gribkovskaia, and G.~Laporte, ``Static pickup
  and delivery problems: a classification scheme and survey,'' {\em Top},
  vol.~15, no.~1, pp.~1--31, 2007.

\bibitem{psaraftis1995dynamic}
H.~N. Psaraftis, ``Dynamic vehicle routing: Status and prospects,'' {\em Annals
  of operations research}, vol.~61, no.~1, pp.~143--164, 1995.

\bibitem{pillac2013review}
V.~Pillac, M.~Gendreau, C.~Gu{\'e}ret, and A.~L. Medaglia, ``A review of
  dynamic vehicle routing problems,'' {\em European Journal of Operational
  Research}, vol.~225, no.~1, pp.~1--11, 2013.

\bibitem{buchheim2012effective}
C.~Buchheim, A.~Caprara, and A.~Lodi, ``An effective branch-and-bound algorithm
  for convex quadratic integer programming,'' {\em Mathematical programming},
  pp.~1--27, 2012.

\bibitem{rahmani2016column}
N.~Rahmani, B.~Detienne, R.~Sadykov, and F.~Vanderbeck, ``A column generation
  based heuristic for the dial-a-ride problem,'' in {\em International
  Conference on Information Systems, Logistics and Supply Chain (ILS)}, 2016.

\bibitem{desrochers1991improvements}
M.~Desrochers and G.~Laporte, ``Improvements and extensions to the
  miller-tucker-zemlin subtour elimination constraints,'' {\em Operations
  Research Letters}, vol.~10, no.~1, pp.~27--36, 1991.

\bibitem{gendreau1997covering}
M.~Gendreau, G.~Laporte, and F.~Semet, ``The covering tour problem,'' {\em
  Operations Research}, vol.~45, no.~4, pp.~568--576, 1997.

\bibitem{savelsbergh1995general}
M.~W. Savelsbergh and M.~Sol, ``The general pickup and delivery problem,'' {\em
  Transportation science}, vol.~29, no.~1, pp.~17--29, 1995.

\bibitem{lawler1963quadratic}
E.~L. Lawler, ``The quadratic assignment problem,'' {\em Management science},
  vol.~9, no.~4, pp.~586--599, 1963.

\bibitem{sahai2017continuous}
T.~Sahai, S.~Klus, and M.~Dellnitz, ``Continuous relaxations for the traveling
  salesman problem,'' {\em arXiv preprint arXiv:1702.05224}, 2017.

\bibitem{turner1968generalized}
J.~Turner, ``Generalized matrix functions and the graph isomorphism problem,''
  {\em SIAM Journal on Applied Mathematics}, vol.~16, no.~3, pp.~520--526,
  1968.

\bibitem{winston2003introduction}
W.~L. Winston, M.~Venkataramanan, and J.~B. Goldberg, {\em Introduction to
  mathematical programming}, vol.~1.
\newblock Thomson/Brooks/Cole Duxbury; Pacific Grove, CA, 2003.

\bibitem{burer2012non}
S.~Burer and A.~N. Letchford, ``Non-convex mixed-integer nonlinear programming:
  A survey,'' {\em Surveys in Operations Research and Management Science},
  vol.~17, no.~2, pp.~97--106, 2012.

\bibitem{fletcher1998numerical}
R.~Fletcher and S.~Leyffer, ``Numerical experience with lower bounds for miqp
  branch-and-bound,'' {\em SIAM Journal on Optimization}, vol.~8, no.~2,
  pp.~604--616, 1998.

\bibitem{bonami2008algorithmic}
P.~Bonami, L.~T. Biegler, A.~R. Conn, G.~Cornu{\'e}jols, I.~E. Grossmann, C.~D.
  Laird, J.~Lee, A.~Lodi, F.~Margot, N.~Sawaya, {\em et~al.}, ``An algorithmic
  framework for convex mixed integer nonlinear programs,'' {\em Discrete
  Optimization}, vol.~5, no.~2, pp.~186--204, 2008.

\bibitem{bonami2012algorithms}
P.~Bonami, M.~Kilin{\c{c}}, and J.~Linderoth, ``Algorithms and software for
  convex mixed integer nonlinear programs,'' {\em Mixed integer nonlinear
  programming}, pp.~1--39, 2012.

\bibitem{takapoui2017simple}
R.~Takapoui, N.~Moehle, S.~Boyd, and A.~Bemporad, ``A simple effective
  heuristic for embedded mixed-integer quadratic programming,'' {\em
  International Journal of Control}, pp.~1--11, 2017.

\bibitem{OpenStreetMap}
{OpenStreetMap contributors}, ``{Planet dump retrieved from
  http://planet.osm.org }.'' \url{http://www.openstreetmap.org}, 2018.

\bibitem{bell1965gershgorin}
H.~E. Bell, ``Gershgorin's theorem and the zeros of polynomials,'' {\em The
  American Mathematical Monthly}, vol.~72, no.~3, pp.~292--295, 1965.

\bibitem{haynsworth1968schur}
E.~V. Haynsworth, ``On the schur complement.,'' tech. rep., BASEL UNIV
  (SWITZERLAND) MATHEMATICS INST, 1968.

\end{thebibliography}

\end{document}